\documentclass{amsart}
\usepackage{xcolor}
\usepackage{amssymb,latexsym,amsmath}
\usepackage{graphicx,mathrsfs}

\usepackage[colorlinks,linkcolor=blue,anchorcolor=blue,citecolor=blue]{hyperref}

\makeatletter
\@namedef{subjclassname@2020}{%
  \textup{2020} Mathematics Subject Classification}
\makeatother

\numberwithin{equation}{section}
\numberwithin{figure}{section}

\newtheorem{theorem}{Theorem}[section]
\newtheorem{lemma}[theorem]{Lemma}
\newtheorem{proposition}[theorem]{Proposition}
\newtheorem{remark}[theorem]{Remark}

\newtheorem{definition}[theorem]{Definition}
\newtheorem{corollary}[theorem]{Corollary}

\newtheorem{question}[theorem]{Question}

\newcommand{\e}{\varepsilon}

\newcommand{\wh}{\widehat}

\newcommand{\Id}{{\bf 1}}

\allowdisplaybreaks[4]

\begin{document}

\title[]{Boundedness of Multilinear Hilbert Transforms Along Moment Curves}

\author{Jingwei Guo}
\address{
School of Mathematical Sciences\\
University of Science and Technology of China\\
Hefei, 230026\\ P.R. China}
\email{jwguo@ustc.edu.cn}
    
\author{Xiaochun Li}
\address{
Department of Mathematics\\
University of Illinois at Urbana-Champaign\\
Urbana, IL, 61801, USA}
\email{xcli@illinois.edu}

\author{Guoqing Zhan}
\address{
School of Mathematical Sciences\\
University of Science and Technology of China\\
Hefei, 230026\\ P.R. China}
\email{zhanguoqing@mail.ustc.edu.cn}

\date{}

\thanks{}



\begin{abstract}
We prove the boundedness of multilinear Hilbert transforms along moment curves by establishing uniform estimates for both translated and untranslated multilinear paraproducts and by applying a Sobolev smoothing inequality.   As a byproduct of our method, we obtain uniform bounds for both multilinear Hilbert transforms and multilinear maximal functions along moment curves.
\end{abstract}

\maketitle


\section{Introduction} \label{intro}

Let $n\geq 2$ be an integer and $\Gamma_n=(t^{m_1}, \dots, t^{m_n})$ be a moment curve, where 
the exponents $m_1,\dots, m_n\in \mathbb N$ are  distinct. Define the $n$-linear Hilbert transform associated with 
the moment curve $\Gamma_n$ by 
\begin{equation*}
H_{\Gamma_n}(f_1, \dots, f_n)(x) = \mathrm{p.v.} \int_{\mathbb R} f_1(x-t^{m_1})\cdots f_n(x-t^{m_n}) \,\frac{\mathrm{d}t}{t}\,,
\end{equation*}
for $x\in\mathbb R$  and Schwartz functions $f_1, \dots, f_n$.

\begin{theorem}\label{thm}
Let $n\geq 2$. 
The multilinear Hilbert transform $H_{\Gamma_n}$ satisfies 
\begin{equation*}
 \big\| H_{\Gamma_n}(f_1, \dots, f_n)\big\|_r\lesssim \|f_1\|_{p_1}\cdots\|f_n\|_{p_n}\,,
\end{equation*}
for all $p_1, \dots, p_n\in (1,\infty)$ and $r\in (1/2, \infty)$  whenever
\begin{equation}\label{homo-r}
 \frac{1}{r} =\frac{1}{p_1}+\dots + \frac{1}{p_n}\,.  
\end{equation}
\end{theorem}

In Theorem \ref{thm}, the condition $r>1/2$ is known to be sharp when $n=2$. For $n\geq 3$, the threshold $r=1/2$ is essentially sharp in the following sense: in the nontrivial case where at least one of the exponents \(m_j\) is odd, for every $1/n<r<1/2$, there exist $p_1,\dots,p_n\in(1,\infty)$ satisfying \eqref{homo-r} such that $H_{\Gamma_n}$ is unbounded from $L^{p_1}\times\cdots\times L^{p_n}$ to $L^r$. The corresponding counterexamples for $n\geq 3$ are provided in Appendix \ref{appendixB}.  For $n=2$, Theorem \ref{thm} was proved in \cite{DD}, while the case $n=3$ was established in \cite{HL25}. Our method not only recovers these earlier results but also differs fundamentally from the previous approaches, allowing it to extend to the general multilinear setting.

The proof relies on the following Sobolev smoothing inequality, which is obtained by combining the high-dimensional Sobolev estimate \cite[Theorem 6.1]{KMPWJ} of Kosz, Mirek, Peluse, Wan, and Wright  with a scaling and projection argument from Becker and Krause \cite{BK26}, followed by integration by parts.

\begin{proposition}\label{prop-0}
Let $\rho$ be a smooth function supported on $[-2, 2]$, and
\begin{equation}\label{defT_k}
T_k(f_1, \dots, f_n)(x)=\int f_1(x-t^{m_1}) \cdots f_n(x-t^{m_n}) 2^{k}\rho(2^{k}t)\,\mathrm{d}t\,,
\end{equation}
where $m_1, \dots, m_n$ are distinct positive integers and $k$ is an arbitrary integer. 
Suppose that, for some $j\in\{1,\dots,n\}$, the Fourier transform
$\widehat{f_j}$ is supported in the dyadic annulus 
$\{\xi\in\mathbb{R} : 2^{m_j k}\delta\leq |\xi|\leq 2^{m_j k+1}\delta\}$ with $\delta\geq 1$.  Then 
for any exponents $p_1, \dots, p_n, r\in [1, \infty)$ satisfying the H\"older relation \eqref{homo-r}, there exists $c=c(n,p_{1},\ldots,p_{n},m_{1},\ldots,m_{n})>0$ such that
\begin{equation*}
\big\| T_k(f_1, \dots, f_n)\big\|_r\lesssim \delta^{-c} \prod_{j=1}^n \|f_j\|_{p_j}\,. 
\end{equation*}
\end{proposition}

Using Proposition \ref{prop-0}, we reduce the problem to obtaining uniform estimates for multilinear paraproducts, as demonstrated in Section \ref{To-Para}. We now describe the multilinear paraproducts that arise in our study.  For $j\in\{1, \dots, n\}$, let $\delta_j$ be a (positive) dyadic number.  
Let $\Phi:\mathbb R\rightarrow \mathbb C$ be a Schwartz function whose Fourier transform is supported in $\{\xi\in\mathbb R: 
1/2<|\xi|<2\}$. We define $\Phi_{j, k, \delta}$ by  
\begin{equation}\label{s1-2}
 \wh{\Phi_{j, k, \delta}}(\xi) =\wh\Phi\big( \frac{\xi}{2^{m_j k}\delta} \big)\,.
\end{equation}
Let $\Psi:\mathbb R\rightarrow\mathbb C$ be a Schwartz function whose Fourier transform is supported in $\{\xi\in\mathbb R: 
|\xi|<2\}$. Similarly, we 
define $\Psi_{j, k, \delta}$ by  
\begin{equation*}
 \wh{\Psi_{j, k, \delta}}(\xi) =\wh\Psi\big( \frac{\xi}{2^{m_j k}\delta} \big)\,.
\end{equation*}
For any locally integrable function $f_j:\mathbb R\rightarrow \mathbb C$, we set
\begin{equation*}
 f_{j, k, \delta_j, \Psi} (x) = f_j* \Psi_{j, k, \delta_j}(x)\,.
 \end{equation*}
 Let $n_j$ be a positive integer and $t$ be a real number. We define 
 \begin{equation*}
  f_{j, k, \delta_j, n_j, t}(x) = f_j*\Phi_{j, k, \delta_j} \big(x -2^{n_j} 2^{-m_jk}\delta_j^{-1} t^{m_{j}} \big)\,.
 \end{equation*}
The paraproduct that we encounter is given as follows: for any $J\in \{0, 1, \dots, n\}$, 
\begin{equation}\label{def-Pi}
 \Pi_{\delta_1, \dots, \delta_J} (f_1, \dots\!, f_n)(x) = \sum_{k\in\mathbb K} \int\!\! \rho^*(t) \!\!\!\prod_{1\leq j\leq J} \!\!f_{j, k, \delta_j, n_j, t}(x)  \,\mathrm{d}t \!\!\!\! \prod_{J+1\leq j\leq n}\!\!\!\!
 f_{j, k, \delta_j, \Psi}(x)\,,
\end{equation}
where  $\mathbb K$ is a subset of $\mathbb Z$ and the function $\rho^*$ is a Schwartz function on $\mathbb R$ supported in $\{t\in\mathbb R: 1/2<|t|<2\}$. For $J\in\{2, \dots, n\}$, we call $\Pi_{\delta_1, \dots, \delta_J}$ a $J$-{\it translated paraproduct}, since  the first 
$J$ factors involve the $t$-dependent spatial translations 
$$
x -2^{n_j} 2^{-m_jk}\delta_j^{-1}t^{m_{j}} \,.
$$
Note that if $J=0$, the tuple $(\delta_1,\dots,\delta_J)$ is empty.  In this degenerate case, we denote by $\Pi_0$ the {\it untranslated paraproduct}
$$
\Pi_0=\sum_{k\in\mathbb K}  \prod_{1\leq j\leq n}
 f_{j, k, \delta_j, \Psi}(x)\,.$$
To use uniform notation for all possible values of $J$, we set
$$
  \Pi_{\delta_1, \dots, \delta_J} = \Pi_0\,\,\, {\rm if}\,\,J=0\,.
 $$
Moreover, notice that the paraproduct $\Pi_{\delta_1,\dots,\delta_J}$ also depends on $\delta_{J+1},\dots,\delta_n$ and $m_1,\dots,m_n$. Here, we suppress these dependencies, as $J$ is the most important parameter for the $J$-translated paraproducts. For applications, we need to establish uniform estimates that are independent of $\delta_1,\dots,\delta_n$, but may depend on $m_1,\dots,m_n$.
We remark that, in our paper, we do not need to consider the case $J=1$, since it can be reduced to the case $J=0$, as explained in Section \ref{To-Para}.\\

The study of paraproducts plays a fundamental role in harmonic analysis and partial differential equations (PDEs), and a substantial body of literature has been devoted to their boundedness and mapping properties in various function spaces. In particular, paraproducts provide an important framework for understanding nonlinear interactions between functions at different frequency scales and have become a central tool  in the analysis of nonlinear PDEs.

A number of important results have been established for untranslated paraproducts. For instance, in the case $m_1=\cdots=m_n$, Muscalu, Tao, and  Thiele established uniform $L^r$ estimates, for $r>1$, for untranslated paraproducts in \cite{MTT02-2}. Their approach relies on techniques arising from graph theory and provides a systematic framework for controlling the multilinear interactions associated with these operators. These results constitute an important part of the theory of multilinear paraproducts in the Banach range.

The quasi-Banach range $r<1$, however, presents substantially different difficulties, since $L^r$ is no longer a Banach space and the usual tools based on the triangle inequality are no longer directly available. In this direction, Grafakos, He, Kalton, and  Mastyło studied $L^r$ estimates for regular untranslated paraproducts with $r<1$ in \cite{GHKM}. Their analysis makes use of Hardy space $H^p$ theory and develops techniques adapted to the quasi-Banach setting. These results demonstrate that, although the range $r<1$ requires different methods, appropriate Hardy space and maximal function techniques can still yield strong boundedness results for paraproducts.

The paraproducts considered in the present paper are substantially more challenging, as we study both translated and untranslated paraproducts that are different from those considered in the aforementioned works.
 The bilinear theory of such translated paraproducts, as well as that of the bilinear Hilbert transform along the moment curve, was first studied by the second author in \cite{Li1,Li2}. In the present paper, we develop the corresponding multilinear $L^r$ theory, which forms a central part of our approach and constitutes one of the main contributions of this work.

More precisely, we introduce a new algorithm that yields the full range $r>1/n$ for the 
(irregular) untranslated paraproducts, as defined later in the paper, as well as for the $2$-translated paraproducts.  For the $J$-translated paraproducts with $J\geq 3$, the Banach range can be treated using a new translated Littlewood--Paley theorem established in Section \ref{VLP}. In the quasi-Banach setting, however, our method only yields the range $r>1/2$ for all $3\leq J\leq n$. Thus, while the Banach range can be handled in full generality, the quasi-Banach range remains more delicate. \\

The use of the Sobolev smoothing inequality provides a new perspective on these estimates. In particular, it allows us to exploit the smoothing effect at the level of the multilinear paraproducts and leads to a proof that is fundamentally different from those in the previous works \cite{DD, HL, HL25, Li1, Li2, Li-Xiao}. While these earlier works focused only on the bilinear or trilinear setting, our approach extends the analysis to the multilinear setting.  Thus, the combination of the smoothing inequality with the multilinear paraproduct analysis is a key ingredient in our proof and provides a new framework for establishing the estimates needed in this paper.\\

Our method can also be used to obtain uniform estimates for both the multilinear Hilbert transform and the multilinear maximal function. More precisely, let $\alpha_1,\dots,\alpha_n$ be nonzero real numbers. For distinct exponents $m_1,\dots, m_n\in \mathbb N$, define the $n$-linear Hilbert transform associated with the moment curve $\Gamma_n$ and $(\alpha_1,\dots,\alpha_n)$ by
\begin{equation*}
H_{\Gamma_n, \alpha_1, \dots, \alpha_n}(f_1, \dots, f_n)(x) = \mathrm{p.v.} \int_{\mathbb R} f_1(x-\alpha_1t^{m_1})\cdots f_n(x-\alpha_n t^{m_n}) \,\frac{\mathrm{d}t}{t}\,,
\end{equation*}
for $x\in\mathbb R$ and Schwartz functions $f_1, \dots, f_n$. Similarly, we define the multilinear maximal function associated with $\Gamma_n$ and $(\alpha_1,\dots,\alpha_n)$ by
\begin{equation*}
M_{\Gamma_n, \alpha_1, \dots, \alpha_n}(f_1, \dots, f_n)(x) = \sup_{\e>0} \frac{1}{2\e} \int_{-\e}^{\e} \big| f_1(x-\alpha_1 t^{m_1})\cdots f_n(x-\alpha_n t^{m_n})\big| \,\mathrm{d}t\,. 
\end{equation*}

\begin{theorem}\label{thm-al}
Let $n\geq 2$.  The multilinear Hilbert transform $H_{\Gamma_n, \alpha_1, \dots, \alpha_n}$ satisfies 
\begin{equation*}
\sup_{\alpha_1, \dots, \alpha_n{\neq 0}} \big\| H_{\Gamma_n, \alpha_1, \dots, \alpha_n}(f_1, \dots, f_n)\big\|_r\lesssim \|f_1\|_{p_1}\cdots\|f_n\|_{p_n}\,,
\end{equation*}
for all {$p_1, \dots, p_n\in (1,\infty)$ and $r\in (1/2, \infty)$} whenever
$$
 \frac{1}{r} =\frac{1}{p_1}+\dots + \frac{1}{p_n}\,.  
$$
The same uniform boundedness result holds for the multilinear maximal function $M_{\Gamma_n,\alpha_1,\dots,\alpha_n}$ for all {$p_1, \dots, p_n\in (1,\infty]$ and $r\in (1/2, \infty]$ satisfying the H\"older relation.} 
\end{theorem}

We omit the proof of Theorem \ref{thm-al}, as it is essentially the same as that of Theorem \ref{thm}. 
Here we point out two minor modifications necessitated by the presence of the
coefficients $\alpha_j$. First, Proposition~\ref{prop-0} extends to include coefficients
$\alpha_j$ by applying a scaling argument to the high-dimensional Sobolev
estimate in \cite[Theorem~6.1]{KMPWJ}, provided that $\widehat{f_j}$ is
instead supported in
\[
\left\{
    \xi\in\mathbb{R} :
    |\alpha_j|^{-1}2^{m_jk}\delta
    \leq |\xi|
    \leq |\alpha_j|^{-1}2^{m_jk+1}\delta
\right\}.
\]
Both the implicit constant and the exponent $c$ in the resulting estimate
are independent of the coefficients $\alpha_j$. Second, to apply this
modified Sobolev smoothing inequality, we make the corresponding adjustments
to the frequency decompositions in Section~\ref{To-Para}. All the main steps, including the relevant decomposition and estimates, remain unchanged, with only minor notational modifications. We leave the details to the reader.  Moreover, 
for $n\geq 3$ the threshold $1/2$ in Theorem \ref{thm-al} is sharp  up to the endpoint. Indeed, for the curve $(t^{m_1}/m_1,\dots,t^{m_n}/m_n)$ the boundedness conclusion fails whenever $r<1/2$, as shown by the counterexample constructed in Appendix \ref{appendixA}. For the multilinear Hilbert transform, the counterexample requires at least one of the exponents $m_j$ to be odd, while for the multilinear maximal function no such condition is needed.

In light of the examples provided in Appendices \ref{appendixA} and \ref{appendixB}, it is natural to ask the following question. 

\begin{question}\label{conj-poly}
Given a polynomial curve $\Gamma_n$, what is the optimal threshold $r_{\Gamma_n}$ such that the associated multilinear Hilbert transform $H_{\Gamma_n}$ is bounded from
$
L^{p_1}\times\cdots\times L^{p_n}
\quad\text{to}\quad
L^r
$
for all $p_1,\ldots,p_n>1$ and all $r>r_{\Gamma_n}$ satisfying the H\"older relation \eqref{homo-r}?
\end{question}

In the bilinear setting, Li and Xiao \cite{Li-Xiao} completely resolved this question for polynomial curves of the form $(t,P(t))$, where $P$ has neither a constant nor a linear term.  For general polynomial curves, however, the question remains open even in the bilinear case. The multilinear case $n\geq 3$ appears to be substantially more difficult and remains a subject of ongoing research.
    \\

As a final remark, our approach is not restricted to integer exponents \(m_j\). In principle, it can be extended to certain positive noninteger exponents \(\beta_j\) for which \(t^{\beta_j}\) is real-valued for every \(t\in\mathbb{R}\backslash\{0\}\), such as \(\beta_j=1/3\). Consequently, Theorem \ref{thm} can be extended to multilinear Hilbert transforms along curves of the form
$$
\bigl(\alpha_1t^{\beta_1},\dots, \alpha_n t^{\beta_n}\bigr)
$$
for suitable choices of the distinct exponents $\beta_1, \dots, \beta_n$. \\

\textit{Notation:} We use $A\lesssim B$ to denote the statement that
$A\leq CB$ for an unimportant positive constant $C$. We write $A\lesssim_a B$
to mean that $A\leq C_aB$ for some positive constant $C_a$ that may
depend on $a$. We use $A\preceq_K B$ to denote
\[
A\leq 2^{-K}B.
\]
We write $A\sim B$ if there exist absolute positive constants $C_1<1$
and $C_2>1$ such that
$ 
C_1 A\leq B\leq C_2 A$. 
For two functions $f$ and $g$, we write $f\approx g$ if $f$ and $g$
have essentially the same properties, in the sense that a method
used to handle one can also be applied to the other.

 \textit{The paper is organized as follows.}
The precise statements of the uniform estimates are provided in Section \ref{defPhi-jk}.
In Section \ref{single}, we establish estimates for a single-scale operator. This result allows us to obtain $L^r$ estimates for certain values of $r<1$ in the Sobolev smoothing inequality. This step is essential for reducing the boundedness problem for multilinear Hilbert transforms to the corresponding paraproduct estimates, as demonstrated in Section \ref{To-Para}. In Section \ref{VLP}, we establish a variant of the classical Littlewood--Paley square-function and maximal function estimates that will be used for the translated paraproducts. Section \ref{Banach} is devoted to the study of paraproducts in the Banach range, while Section \ref{Pf-para} is devoted to the corresponding analysis in the quasi-Banach range. In the appendices, we show the sharpness of Theorems \ref{thm} and \ref{thm-al} via counterexamples.

\section{Uniform estimates for the araproducts}\label{defPhi-jk}

In this section, we formulate the uniform estimates for the translated and untranslated paraproducts. Before doing so, let us introduce a decomposition of the underlying index set $\mathbb K$ into so-called {\it well-distributed} sets. This decomposition helps organize the distribution of the Fourier transforms of the functions appearing in the paraproducts, thereby simplifying our analysis.

\begin{definition}
Let $n\geq 2$, $K\geq 1$ be a real number, $m_1, \dots, m_n$ be distinct positive integers, and $\delta_1, \dots, \delta_n$ be dyadic numbers.  A subset $\mathbb K$ of $\mathbb Z$ is called well-distributed of order $K$ if either its cardinality is $O_n(K)$,  or there exists a permutation $(i_1,\dots,i_n)$ of
$(1,\dots,n)$ such that
\begin{equation*}
2^{m_{i_1}k}\delta_{i_1}
\preceq_K 2^{m_{i_2}k}\delta_{i_2} \preceq_K \cdots
\preceq_K 2^{m_{i_n}k}\delta_{i_n} \qquad\text{for every } k\in\mathbb K.
\end{equation*}
The latter is called a non-degenerate well-distributed set, while the former is called a degenerate well-distributed set.
\end{definition}

Using this definition, we obtain the following theorem.

\begin{theorem}\label{de-K}
Any subset $\mathbb K$ of $\mathbb Z$ can be partitioned into $O_n(1)$ subsets that are well-distributed of order $K$. 
\end{theorem}

\begin{proof}
We associate to each $k$ a tuple $ (2^{m_1k}\delta_1, \dots, 2^{m_nk}\delta_n )$ (whose components actually correspond to the scales of the supports of the Fourier transforms of $f_{j, k, \delta_j}$ defined in \eqref{s4-1}).  
We will classify the collection of all $k\in \mathbb{K}$ according to the properties of these tuples. 
For any two distinct indices $i_1, i_2\in \{1, \dots, n\}$, define 
\begin{equation*}
 \mathbb K(i_1, i_2):=\left\{ k\in\mathbb{K} :    2^{-K}2^{m_{i_1}k}\delta_{i_1}
 \leq 2^{m_{i_2}k}\delta_{i_2}\leq 2^{K} 2^{m_{i_1}k}\delta_{i_1} \right\}\,.
\end{equation*}
Since the $\delta_{j}$ are dyadic numbers and the $m_j$ are distinct,  we observe that 
\begin{equation}\label{keyK-0}
\#\mathbb K(i_1, i_2)\leq 1+\frac{2K}{|m_{i_1}-m_{i_2}|}\lesssim K\,.
\end{equation}
We partition $\mathbb K$, the index set of $k$, into
$$
 \mathbb K = \bigcup_{\substack{ i_1, \dots, i_n\in\{1, \dots, n\}\\ 
 i_1,\dots, i_n\, {\rm distinct}}} \mathbb K[i_1, \dots, i_n] \cup  \mathbb K'\,,
$$
where $  \mathbb K[i_1, \dots, i_n]$ is defined by
\begin{equation*}
 \mathbb K[i_1, \dots, i_n]:= \left\{ k\in\mathbb{K} :  2^{m_{i_1}k}\delta_{i_1} \preceq_K  2^{m_{i_2}k}\delta_{i_2} \preceq_K  \dots \preceq_K 
  2^{m_{i_n}k}\delta_{i_n} \right\}\,,
 \end{equation*}
and $\mathbb K'$ denotes  the complement of the union of the sets $\mathbb K[i_1, \dots, i_n]$.
In view of the observation \eqref{keyK-0}, it follows that $\mathbb K'$ is a finite set
whose cardinality is bounded by $C_n K$, where 
$C_n$ is  a constant depending only on $n$.  This completes the proof.
\end{proof}

Theorem \ref{de-K} indicates that we only need to consider well-distributed index sets $\mathbb K$.

\begin{definition} 
A collection of functions $\big\{f_k\big\}_{k\in\mathbb K}$ is said to be of Littlewood--Paley type with parameters $\alpha$ and $\delta$ if there exist $\alpha\in\mathbb{N}$ and a dyadic number $\delta$ such that, for each $k\in\mathbb K$, the Fourier transform of $f_k$ is supported in the set
\begin{equation*}
    \{\xi\in\mathbb{R} : 2^{\alpha k}\delta\leq |\xi|\leq 2^{\alpha k+1}\delta\}.
\end{equation*}
\end{definition}

\begin{definition}
Let $J\in\{0,2,\dots,n\}$. The paraproduct $\Pi_{\delta_1,\dots,\delta_J}$, defined as in \eqref{def-Pi}, is called admissible if either the underlying index set $\mathbb K$ is finite, with cardinality bounded above by a constant $C_n$, or the paraproduct satisfies the following conditions:
\begin{itemize}
\item  $\mathbb K$ is well-distributed of order $n_J^*+100\max\big\{n, m_1, \dots, m_n\big\}$, where 
$n_J^*$ is defined by
$$
  n_J^* =\max\{0, n_1, \dots, n_J \}\,.  
$$ 

\item  If $J=0$, then there exists $i_1\in \{1, \dots, n\}$ such that $\{f_{i_1, k, \delta_{i_1}, \Psi}\}_{k\in \mathbb K}$ is of Littlewood--Paley type
with parameters $m_{i_1}$ and $\delta_{i_1}$.  For convenience, we also say that the index $i_1$ is of Littlewood--Paley type if $\{f_{i_1,k,\delta_{i_1},\Psi}\}_{k\in\mathbb K}$ is of Littlewood--Paley type, particularly when referring to the index $i_1$ associated with the functions in the paraproduct.
\end{itemize}
\end{definition}

For the admissible paraproducts, we are able to establish the following uniform estimates, whose proofs will be provided in Sections \ref{Banach} and \ref{Pf-para}.

\begin{theorem}\label{para}
For $n\geq 2$, let $ \Pi_{ \delta_1, \dots, \delta_J} $, defined as in \eqref{def-Pi},  be an admissible paraproduct  and $r_{J}$ be defined by 
\begin{equation*}
r_{J} =
\begin{cases}
\frac{1}{n}, & \text{if } J=0 \,\, {\rm or}\,\, J=2;\\[2mm]
\frac{1}{2}, & \text{if } J\geq 3.
\end{cases}
\end{equation*}
 Then, for any $\e>0$,  there is a constant $C$,  independent of 
$\delta_1, \dots, \delta_n$  and $n_1, \dots, n_J$,  such that 
\begin{equation*}
\big\| \Pi_{\delta_1, \dots, \delta_J}  (f_1, \dots, f_n) \big\|_r \leq C 2^{\e n_J^*} \prod_{j=1}^n \big\|f_j\big\|_{p_j}\,,
\end{equation*}
 for all $p_1, \dots, p_n>1$ and $r> r_{J}$, provided that $(p_1, \dots, p_n, r)$ obeys the H\"older relation \eqref{homo-r}.
\end{theorem}

The biggest enemy arises from the translation $x -2^{n_j} 2^{-m_jk}\delta_j^{-1} {t^{m_j}}$.
The corresponding Littlewood--Paley square function is 
\begin{equation*}
\bigg( \sum_{k\in \mathbb K} \big| f_{j, k, \delta_j, n_j, t}(x) \big|^2 \bigg)^{1/2}\,,
\end{equation*}
which satisfies 
$$
\bigg\|\bigg( \sum_{k\in \mathbb K} \big| f_{j, k, \delta_j, n_j, t}\big|^2 \bigg)^{1/2}\bigg\|_p
\lesssim n_j \|f_j\|_p\,, 
$$
for any $p>1$, as proved in Section \ref{VLP} using the standard Calder\'on--Zygmund theory.  This estimate plays an important role in our proof, especially in the case $r\geq 1$. However, an additional difficulty arises when $r<1$. This difficulty is the main reason why our argument yields estimates only in the range $r>1/2$ for $J\geq 3$. 
In view of the counterexamples presented in the appendices, the bound $r>1/n$ should not be regarded as the appropriate goal in general; instead, $r>1/(2+n-J)$ appears to be the natural threshold for paraproducts associated with certain moment curves.

\subsection{Reverse square function estimates}
There is a useful tool from the standard $H^p$-theory that allows us to relate the sum of a Littlewood--Paley-type collection to the Littlewood--Paley square function. Let us state the following reverse square-function estimates, which are consequences of He's Corollary 4 in \cite{He}. We include a proof for the reader’s convenience.

\begin{lemma}\label{RSFE}  
      Let $p\in (0, \infty)$, $\mathbb K\subset \mathbb Z$, and $\{f_k\}_{k\in\mathbb K}$ be of Littlewood--Paley type with parameters $\alpha$ and $\delta$. 
Suppose that $\sum_{k\in\mathbb K} f_k$ converges almost everywhere. Then we have  
\begin{equation}  \label{RSFE-1}
		\bigg\| \sum_{k \in \mathbb K} f_k \bigg\|_{p}
		\lesssim_{p}
		\bigg\| \bigg( \sum_{k \in \mathbb K} |f_k|^2 \bigg)^{1/2} \bigg\|_{p},
	\end{equation}
    where the implicit constant is independent of $\delta$.  
\end{lemma}

\begin{proof}
	 Note that the case $p > 1$ follows from the classical Littlewood--Paley theory for $L^p$ spaces. Here we only prove the case $p \leq 1$. It suffices to consider the case $\mathbb K=\mathbb Z$. For $N\in \mathbb N$, let 
     $$
      F_N=\sum_{|k|\leq N}f_k\,.
     $$
First, we represent 
\begin{equation}\label{al-del}
  \wh{f_k}(\xi) = \wh{f_k}(\xi) \wh\phi\big( \frac{\xi}{ 2^{\alpha k}\delta}\big) \,,    
\end{equation}
 where $\phi:\mathbb R\rightarrow \mathbb C$ is a Schwartz function whose Fourier transform is supported in $\{\xi\in\mathbb R: 
1/2<|\xi|<3\}$ and equals $1$ on $ \{\xi\in\mathbb R:1\leq |\xi|\leq 2\}$.   Then
by the  Peetre--Fefferman--Stein maximal inequality,
we have, for large  $N'\in\mathbb N$, 
$$
\big| f_k(x)\big|\lesssim \int \frac{\big| f_k(y) 2^{\alpha k}\delta \big| }{\big( 1+ 2^{\alpha k}\delta|x-y|\big)^{N'}} \,\mathrm{d}y \lesssim 
 \bigg(M\big( |f_k|^{p/2}\big)\bigg)^{2/p}\,,
$$
where $M$ denotes the Hardy--Littlewood maximal function.  Then, by the vector-valued inequality for $M$ (\cite[Theorem 5.6.6]{GTM249}) in $L^2(\ell^{4/p})$, we have
		\begin{equation}\label{MAX}
		\bigg\| \bigg( \sum_k  \bigg(M\big( |f_k|^{p/2}\big)\bigg)^{4/p}\bigg)^{1/2} \bigg\|_p
		\lesssim 
		\bigg\| \bigg( \sum_k |f_k|^2 \bigg)^{1/2} \bigg\|_p.
	\end{equation}
Moreover, notice that 
\begin{equation}\label{Del}
    \Delta_k F_N=\sum_{|\ell|\leq C_0}\Delta_k f_{k+\ell}\,,
\end{equation}
where $C_0$ is some absolute constant and $\Delta_k f$ is defined by $\wh{\Delta_k f}(\xi)= \wh{f}(\xi) \wh\phi( \frac{\xi}{ 2^{\alpha k}\delta}) $.
Combining \eqref{Del} with \eqref{MAX}, we obtain
\begin{equation}\label{DEL-1}
		\bigg\| \bigg( \sum_k  \big |\Delta_k F_N\big|^{2}\bigg)^{1/2} \bigg\|_p
		\lesssim 
		\bigg\| \bigg( \sum_k |f_k|^2 \bigg)^{1/2} \bigg\|_p.
	\end{equation}

	By the Littlewood--Paley characterization of $H^p$ (\cite[Theorem 2.2.9]{GTM250}) and the continuous embedding $H^p \hookrightarrow L^p$, there exists a unique polynomial $Q$ such that $F_N - Q \in H^p$, and
		\begin{equation*}
		\big\|F_N - Q\big\|_{L^p} \lesssim_p \big\|F_N - Q\big\|_{H^p} 
		\lesssim \bigg\| \bigg( \sum_k |\Delta_k F_N|^2 \bigg)^{1/2} \bigg\|_p.
	\end{equation*}
Since the right-hand side of \eqref{RSFE-1} can be assumed finite without loss of generality, we have 
$Q=0$, and hence the desired estimate \eqref{RSFE-1}  follows from  \eqref{DEL-1} and Fatou's lemma. 
\end{proof}

In applications, the number $\alpha$ will take only the finitely many positive integer values $m_1$, \dots, $m_n$.

We also have a similar result for a certain maximal operator that will arise in our study of the paraproduct.

\begin{lemma}\label{MaxRS}  
      Let $p\in (0, \infty)$, $\mathbb K\subset \mathbb Z$, $\alpha_1\in \mathbb{N}$, and $\delta_1\in 2^{\mathbb Z}$. Let $\{f_k\}_{k\in\mathbb K[\alpha_1, \delta_1]}$ be of Littlewood--Paley type with parameters $\alpha$ and $\delta$, where the index set $\mathbb K[\alpha_1, \delta_1]$ is defined by
      $$
      \mathbb K[\alpha_1, \delta_1] =\big\{ k\in\mathbb K:  2^{\alpha k}\delta \preceq_K  2^{\alpha_1 k}\delta_{1}\big\} 
      $$
    for sufficiently large $K$. Suppose that $\sum_{k\in\mathbb K[\alpha_1, \delta_1]} f_k$ converges almost everywhere. Then we have  
\begin{equation}  \label{MaxRS-1}
		\Bigg\| \sup_{h\in\mathbb Z} \bigg|  \sum_{\substack{k \in \mathbb K[\alpha_1, \delta_1] \\  h\geq \alpha_1 k +\log_2\delta_1}} f_k \bigg|\Bigg\|_p
		\lesssim_{p}
		\bigg\| \bigg( \sum_{k \in \mathbb K[\alpha_1, \delta_1]} |f_k|^2 \bigg)^{1/2} \bigg\|_{p},
	\end{equation}
    where the implicit constant is independent of $\delta$ and $\delta_1$.  
\end{lemma}

\begin{proof}
    Write $I=\mathbb K[\alpha_1,\delta_1]$ and first assume that $I$ is finite. By splitting $I$ into its even and odd parts and using the $L^p$ quasi-triangle inequality, it suffices to prove the estimate when all indices in $I$ have the same parity. We may therefore assume that distinct indices in $I$ are separated by at least two.

    Choose a fixed Schwartz function $\Psi$ such that
\[
 \widehat\Psi(\xi)=1\quad\text{for }|\xi|\leq2,
 \qquad \operatorname{supp}\widehat\Psi\subset(-4,4),
\]
and write $\Psi_t(x)=t\Psi(tx)$ for $t>0$. In particular,
$\int_\mathbb R\Psi(x)\,\mathrm{d}x=1$. Fix $h\in\mathbb Z$. If
$\{k\in I:h\geq\alpha_1k+\log_2\delta_1\}$ is nonempty, let
\[
 \kappa=\max\{k\in I:h\geq\alpha_1k+\log_2\delta_1\},
 \qquad t=2^{\alpha\kappa}\delta.
\]
 Then
\begin{equation*} 
 \sum_{\substack{k\in I\\ h\geq\alpha_1k+\log_2\delta_1}}f_k
       =\Psi_t*\Big(\sum_{k\in I}f_k\Big).
\end{equation*}
If the selected set is empty, its sum is zero. Then we have 
$$
\bigg\| \sup_{h\in\mathbb Z} \bigg|  \sum_{\substack{k \in I \\  h\geq \alpha_1 k +\log_2\delta_1}} f_k \bigg|\bigg\|_{p}\leq
\bigg\|\sup_{t>0} \big| \Psi_{t} *\big(\sum_{k\in I} f_k\big)\big| \bigg\|_p \lesssim \bigg\|\sum_{k\in I} f_k \bigg\|_{H^p}\,,
$$
from which the desired estimate \eqref{MaxRS-1} then follows from the proof of Lemma \ref{RSFE}.  

For infinite $I$, applying the above estimate to $I\cap[-N,N]$ and then using Fatou’s lemma yields the desired bound.
\end{proof}



\section{Single-scale estimates} \label{single}

In this section, we study a single-scale estimate that helps us understand the paraproduct when the underlying index set is a degenerate well-distributed set.   Define 
\begin{equation*}
 T_{ \alpha_1, \dots, \alpha_n, k}(f_1, \dots, f_n)(x)= \int f_1(x-\alpha_1 t^{m_1}) \cdots f_n(x-\alpha_n t^{m_n}) 2^k \rho(2^k t) \,\mathrm{d}t\, ,
\end{equation*}
where $k\in\mathbb{Z}$, $\alpha_1,\dots,\alpha_n\in \mathbb{R}$ and $\rho$ is a smooth function supported in $\{t\in\mathbb R: |t|\sim 1 \}$. 
For $p_1, \dots, p_n\geq 1$, recall that $$
 \frac{1}{r} = \frac{1}{p_1} +\dots +\frac{1}{p_n}\,.
$$
We aim to establish the following uniform $(p_1, \dots, p_n, r)$-estimate
\begin{equation}\label{de}
 \big\| T_{ \alpha_1, \dots, \alpha_n, k}(f_1, \dots, f_n)\big\|_{r}\leq C \|f_1\|_{p_1}\cdots\|f_n\|_{p_n}\,,
\end{equation}
where the constant $C$ is independent of $k$, $f_j$, and $\alpha_j$.  
It is easy to get the desired estimates for $r\geq 1$. In fact, we have

\begin{lemma}\label{r>1}
The uniform estimate \eqref{de} holds when $r\geq 1$. 
\end{lemma}

\begin{proof}
When $r\geq 1$, we apply Minkowski's inequality, which essentially is the triangle inequality in $L^p$, to obtain 
$$
\big\| T_{\alpha_1, \dots, \alpha_n, k}(f_1, \dots, f_n)\big\|_{r}\leq \int \big\| f_1(x-\alpha_1t^{m_1})\cdots f_n(x-\alpha_n t^{m_n})\big\|_{L^r_x} 2^k \big| \rho(2^k t) \big| \,\mathrm{d}t\,.
$$
By using H\"older's inequality,  the $L^r$-norm in the integrand can be dominated by 
$$
  \big\| f_1(x-\alpha_1t^{m_1})\cdots f_n(x-\alpha_n t^{m_n})\big\|_{L^r_x} \leq \|f_1\|_{p_1}\cdots\|f_n\|_{p_n}\,.
 $$
This completes the proof since $\rho\in L^1$. 
\end{proof}

The main difficulty arises when $r<1$.  We now state our results for the sharp range of quasi-Banach exponents.

\begin{theorem}\label{thm1/2}
Suppose that $m_1, \dots, m_n $ are distinct positive integers, and none of $\alpha_1, \dots, \alpha_n$ is zero.  
Then, for $n\geq 2$, the uniform estimate \eqref{de} holds for all $r>1/2$ and $p_1, \dots, p_n\geq 1$. 
\end{theorem}

\subsection{Technical lemmas}

\begin{lemma}\label{bil-lem}
Let $\beta\in[-1,1]$, let $m_1, m_2$ be distinct positive integers, and let $\rho$ be a smooth function supported on $[1/2, 1]\cup [-1, -1/2]$. Define 
\begin{equation*}
T_{ 0, \beta}(f_1, f_2)(x)= \int f_1(x-\beta t^{m_1}) f_2(x-t^{m_2}) \rho(t) \,\mathrm{d}t \,.
\end{equation*}
Then for $1/2< r<1$,  we have the following uniform estimate
\begin{equation}\label{estT0be}
\big\| T_{0, \beta}(f_1, f_2)\big\|_r \leq C\big \|f_1\big\|_1\big\|f_{2}\big\|_{\frac{r}{1-r}}\,,
\end{equation}
where $C$ is independent of $\beta$, but may depend on $m_1$ and $m_2$. 
 \end{lemma}
 
\begin{proof}
By splitting the integral into the regions \(t>0\) and \(t<0\), it suffices to consider kernels supported in either \([1/2,1]\) or \([-1,-1/2]\). We fix one of these two cases throughout the proof.

Since $|t|\sim 1$ and $|\beta|\leq 1$, a localization argument allows us to reduce \eqref{estT0be} to
\begin{equation}\label{estT0be1}
\big\| T_{0, \beta}(f_1, f_2)\big\|_{L^r([0, 1])}\leq C \|f_1\|_1\big\|f_{2}\big\|_{\frac{r}{1-r}}\,.
\end{equation} 

Let 
$$
 u= x-\beta t^{m_1}\,\,\, {\rm and}\,\,\, v= x-t^{m_2}\,.  
$$
Then the Jacobian is equal to 
$$
\frac{\partial (u, v)}{\partial (x, t)}= t^{m_2-1} \big( \beta m_1 t^{m_1-m_2} -m_2\big)\,.
$$
Since $|t|\sim 1$,  if 
\begin{equation}\label{cond-beta}
|\beta|\notin \bigg[ \frac{m_2}{2m_1}2^{-|m_1-m_2|}, \frac{2m_2}{m_1}2^{|m_1-m_2|}\bigg]\,,
\end{equation}
then  
$$
 \bigg| \frac{\partial (u, v)}{\partial (x, t)}\bigg| \geq C_{m_1, m_2} >0\,,
 $$
for some constant $C_{m_1, m_2}>0$ depending on $m_1, m_2$. Hence, when $\beta$ satisfies \eqref{cond-beta}, 
we obtain 
\begin{equation*}
\big\| T_{0, \beta}(f_1, f_2)\big\|_1\leq 
\iint \big| f_1(u) f_2(v) \big| \bigg| \frac{\partial (u, v)}{\partial (x, t)}\bigg|^{-1}
\mathrm{d}u\mathrm{d}v \lesssim_{m_1, m_2} \|f_1\|_1\|f_2\|_1\,.
\end{equation*}

Applying H\"older's inequality and then the last inequality, we get 
$$
\big\| T_{0, \beta}(f_1, f_2)\big\|_{L^r([0, 1])}\leq \big\| T_{0, \beta}(f_1, f_2)\big\|_1\lesssim
\|f_1\|_1\|f_2\|_{L^1([-2^{m_2}, 2^{m_2}])} \lesssim \|f_1\|_1\big\|f_{2}\big\|_{\frac{r}{1-r}}\,,$$
as desired.  

When $\beta$ does not satisfy \eqref{cond-beta}, we shall decompose the range of $t$ according to the size of the 
Jacobian.  For any $j\in \mathbb Z$, we define 
\begin{equation*}
 I_j = \bigg\{ t\in {\rm supp}\rho:   2^{-j-1}\leq \bigg| \frac{\partial (u, v)}{\partial (x, t)}\bigg|\leq 2^{-j}  \bigg\}\,.
\end{equation*}
Since $|\beta|\leq 1$, we see that there exists a constant $c_{m_1, m_2}$ such that 
$I_j =\emptyset$ whenever $j \leq c_{m_1, m_2}$.  Thus we only need to focus on those $j\geq c_{m_1, m_2}$. 
Moreover, it is easy to see that when $\beta$ does not satisfy \eqref{cond-beta},  $I_j$ is contained in an interval of length at most 
$
C_{m_1, m_2} 2^{-j} 
$.
Note that 
\begin{equation*}
T_{ 0, \beta}(f_1, f_2)(x)=\sum_{j\geq c_{m_1, m_2}} T_{0, \beta, j}(f_1, f_2) (x)\,,
\end{equation*}
where
$$
T_{0, \beta, j}(f_1, f_2) (x)= \int f_1(x-\beta t^{m_1}) f_2(x-t^{m_2}) \rho(t)\Id_{I_j}(t) \,\mathrm{d}t \,.
 $$
We aim to  show that 
\begin{equation}\label{est-Tj}
\big\| T_{0, \beta, j}(f_1, f_2)\big\|_{L^r([0, 1])} \lesssim 2^{-\e j} \|f_1\|_1\big\|f_{2}\big\|_{\frac{r}{1-r}}\,,
\end{equation}
where $\e$ is a positive number depending on $r$.  It is clear that \eqref{estT0be1} follows from 
\eqref{est-Tj} because,  for $r<1$, it is easy to see that 
\begin{equation*}
\big\| T_{0, \beta}(f_1, f_2)\big\|_{L^r([0, 1])} 
\leq \bigg( \sum_{j\geq c_{m_1, m_2}} \big\| T_{0, \beta, j}(f_1, f_2)\big\|_{L^r([0,1])}^r\bigg)^{1/r}
\,.
\end{equation*}

Finally we turn to the proof of \eqref{est-Tj}.  
Let $t_j$ be a point in $I_j$.  Set
$$
f_{1, j}(x)=f_1(x-\beta t_j^{m_1})\,\, \, {\rm and}\,\,\, f_{2, j}(x)= f_2(x-t_j^{m_2})\,.
$$
Then 
$$
T_{0, \beta, j}(f_1, f_2) (x)= \int f_{1, j}(x+\beta t_j^{m_1}-\beta t^{m_1}) f_{2,j}(x+t_j^{m_2}-t^{m_2}) \rho(t)\Id_{I_j}(t) \,\mathrm{d}t \,.
$$
Notice that the translations of $f_{1, j}$ and $f_{2, j}$ preserve their $L^p$ norms, and that these functions remain supported in a compact set around the origin  when $x\in [0,1]$ and $t\in {\rm supp} \rho$.  Observe that, for $t\in I_j$, 
both $\beta t_j^{m_1}-\beta t^{m_1} $ and $ t_j^{m_2}-t^{m_2}$ are $O(2^{-j})$. By a localization argument, \eqref{est-Tj}
can be reduced to 
\begin{equation}\label{est-Tj-1}
\big\| T_{\beta,j}(f_1, f_{2})\big\|_{L^r(J_j)} \lesssim 2^{-\e j} \|f_1\|_1\big\|f_{2}\big\|_{\frac{r}{1-r}}\,,
\end{equation}
 where $J_j$ is an interval of length $O(2^{-j})$ and
 $$
  T_{\beta,j}(f_1, f_{2})(x)=  \int f_{1}(x+\beta t_j^{m_1}-\beta t^{m_1}) f_{2}(x+t_j^{m_2}-t^{m_2}) \rho(t)\Id_{I_j}(t) \,\mathrm{d}t \,.   $$
 By inserting absolute values throughout and applying interpolation, we see that $T_{\beta, j}$ maps $L^p\times L^q$ to $L^s$ with 
  a uniform bound $C2^{-j}$, whenever $(1/p, 1/q,1/s)$ belongs to the closed convex hull $\mathbb H$ of the points 
  $(1, 0,1)$, $(0,1,1)$, and $(0,0,0)$. That is, for any $(1/p, 1/q, 1/s)\in \mathbb H$, we have 
  \begin{equation}\label{pqs}
\big\| T_{\beta,j}(f_1, f_{2})\big\|_{L^{s}(J_j)}\lesssim 2^{-j}\|f_1\|_p\|f_2\|_q\,.
\end{equation}

 On the other hand, by H\"older's inequality, we have
 $$
 \big\| T_{\beta,j}(f_1, f_{2})\big\|_{L^{1/2}(J_j)}\leq |J_j|  \big\| T_{\beta,j}(f_1, f_{2})\big\|_{1}  
 \lesssim 2^{-j}  \big\| T_{\beta,j}(f_1, f_{2})\big\|_{1}  \,,$$  
 which is bounded by
 $$
2^{-j} 2^{j} \|f_1\|_1\|f_2\|_1= \|f_1\|_1\|f_2\|_1\,.  $$ 
This yields 
\begin{equation}\label{1/2}
\big\| T_{\beta,j}(f_1, f_{2})\big\|_{L^{1/2}(J_j)}\lesssim \|f_1\|_1\|f_2\|_1.
\end{equation}
Now \eqref{est-Tj-1} follows by interpolation between \eqref{pqs} and \eqref{1/2}.  This  completes the proof of Lemma \ref{bil-lem}.  
\end{proof}

\begin{lemma}\label{inf-bil}
Suppose that $m_1, \dots, m_n $ are distinct  positive integers, and none of $\alpha_1,  \dots, \alpha_n$ is zero.  
For $r\in (1/2, 1)$, 
\begin{equation}\label{T0-r}
\big\| T_{\alpha_1, \dots, \alpha_n, 0}(f_1, \dots, f_n)\big\|_{r} \leq C \|f_1\|_1\|f_2\|_{\frac{r}{1-r}}\prod_{j=3}^n \|f_j\|_\infty\,
\end{equation}
holds for any $f_1\in L^1, f_2\in L^{\frac{r}{1-r}}$, $f_3, \dots, f_n\in L^\infty$. Here the constant $C$ 
is independent of the coefficients $\alpha_j$.  
\end{lemma}

\begin{proof}
\underline{{\bf Case 1}: $|\alpha_1|\leq |\alpha_2|$}.   In this case, notice 
$$
T_{\alpha_1, \dots, \alpha_n,0}(f_1, \dots, f_n)(\alpha_2 x)= \int f_1(\alpha_2 x-\alpha_1 t^{m_1})  \cdots f_n(\alpha_2 x-\alpha_n t^{m_n})  \rho(t) \,\mathrm{d}t\,.
$$
Let 
$$
 f_{1, \alpha_2}(x)= |\alpha_2| f_1(\alpha_2 x) \,\,\, {\rm and} \,\,\, f_{2, \alpha_2}(x)=|\alpha_2|^{\frac{1-r}{r}}f_2(\alpha_2 x)\,.
$$
Clearly we have 
$$
 \big\|f_{1, \alpha_2}\big\|_1=\|f_1\|_1\,\,\, {\rm and}\,\,\, \big\|f_{2, \alpha_2}\big\|_{\frac{r}{1-r}}=\|f_2\|_{\frac{r}{1-r}}\,. 
$$
Moreover, we get 
$$
\big\| T_{ \alpha_1, \dots, \alpha_n, 0}(f_1, \dots, f_n)\big\|_{r}= \big\|T_0^{(\alpha_2)}(f_{1, \alpha_2}, f_{2, \alpha_2}, f_3, \dots, f_n)\big\|_{r}\,,
$$
where $T_0^{(\alpha_2)}(f_{1, \alpha_2}, f_{2, \alpha_2}, f_3, \dots, f_n)(x)$ is defined as
\begin{equation*}
 \int f_{1, \alpha_2}\big( x-\frac{\alpha_1}{\alpha_2} t^{m_1}\big) f_{2, \alpha_2}(x-t^{m_2})
 f_3(\alpha_2 x-\alpha_3t^{m_3}) \cdots f_n(\alpha_2 x-\alpha_n t^{m_n})  \rho(t) \,\mathrm{d}t\,,
 \end{equation*}
whose absolute value is dominated by 
$$
\int \big| f_{1, \alpha_2}\big( x-\frac{\alpha_1}{\alpha_2} t^{m_1}\big) f_{2, \alpha_2}(x-t^{m_2})\big|
 \big|\rho(t)\big| \,\mathrm{d}t \prod_{j=3}^n \big\|f_j\big\|_\infty\,. 
 $$
Since $|\alpha_1/\alpha_2|\leq 1$, Lemma \ref{bil-lem} (whose proof also yields the same estimate with absolute values inserted in the integrand) implies immediately the desired conclusion \eqref{T0-r}.

\underline{\bf Case 2: $|\alpha_2| < |\alpha_1|$}. This case can be handled by first making the change of variables
\(x \mapsto \alpha_1 x\), as in the beginning of Case 1. The rest of the proof
proceeds as before, so we omit the details.

This completes the proof of Lemma \ref{inf-bil}.
\end{proof}

\subsection{Proof of Theorem \ref{thm1/2} }

Since $1/r=1/{p_1}+\dots +1/p_n$,  we can rescale the inequality so that it is sufficient to prove the following uniform estimate, 
\begin{equation}\label{de0}
 \big\| T_{\alpha_1, \dots, \alpha_n,0}(f_1, \dots, f_n)\big\|_{r}\leq C \|f_1\|_{p_1}\cdots\|f_n\|_{p_n}\,,
\end{equation}
where 
$$
T_{\alpha_1, \dots, \alpha_n,0}(f_1, \dots, f_n)(x)= \int f_1(x-\alpha_1 t^{m_1}) \cdots f_n(x-\alpha_n t^{m_n})  \rho(t) \,\mathrm{d}t\,.
$$
 
By the multilinear interpolation theorem,  we only need to verify 
$$
\big\| T_{ \alpha_1, \dots, \alpha_n,0}(\Id_{E_1}, \dots, \Id_{E_n})\big\|_{r}\lesssim \big|E_1\big|^{1/p_1}\cdots\big|E_n\big|^{1/p_n}\,. 
$$
Here we may assume that $r\in (1/2, 1)$ is near $1/2$ because the $r>1$ case has already been handled in Lemma \ref{r>1}.  

We choose two distinct sets $E_{i_1}$ and $E_{i_2}$ among given $E_1, \dots, E_n$ such that
$$
 |E_{i_1}|\leq |E_{i_2}| \leq \min\big\{|E_j|: j\in \{1, \dots, n\}\backslash \{i_1, i_2\} \big\}\,.
$$
We prove a slightly stronger inequality
\begin{equation*}
\big\| T_{ \alpha_1, \dots, \alpha_n,0}(\Id_{E_1}, \dots, \Id_{E_n})\big\|_{r} \lesssim \big| E_{i_1}\big|\big| E_{i_2}\big|^{\frac{1}{r}-1}\,.
\end{equation*}
To see why this inequality implies the desired estimate, observe that 
$$
 \big| E_{i_1}\big|^{1-1/p_{i_1}}  \big| E_{i_2}\big|^{\frac{1}{r}-1-1/p_{i_2}} \lesssim 
 \prod_{j\notin \{i_1, i_2\}} \big| E_j\big|^{1/p_j}\,.
 $$
Without loss of generality,  we may assume $i_1=1$ and $i_2=2$. 
It then suffices to show that 
\begin{equation*}
\big\| T_{ \alpha_1, \dots, \alpha_n,0}(\Id_{E_1}, \dots, \Id_{E_n})\big\|_{r} \lesssim \big| E_{1}\big|\big| E_{2}\big|^{\frac{1}{r}-1}\,,
\end{equation*}
which follows immediately from Lemma \ref{inf-bil}.  Therefore, the proof is complete. \qed


\section{Proof of Theorem \ref{thm}}\label{To-Para}

In this section, we present a proof of our main theorem on the multilinear Hilbert transforms, using uniform estimates for the paraproducts, as given in Theorem \ref{para}, and the Sobolev smoothing inequality in Proposition \ref{prop-0}.
It is more convenient to apply the following proposition, which follows from Proposition \ref{prop-0} and Theorem \ref{thm1/2}.

\begin{proposition}\label{prop1}
Let 
$$
\mathbb P_n=\left\{ \Big(\frac{1}{p_1}, \dots, \frac{1}{p_n}, \frac{1}{r}\Big): \frac{1}{r}=\frac{1}{p_1}+\dots +\frac{1}{p_n}, \,\,
  p_1, \dots, p_n>1, r>1/2
   \right\}\,. 
$$
Suppose that, for some $j\in\{1,\dots,n\}$, the Fourier transform
$\widehat{f_j}$ is supported in the dyadic annulus 
$\{\xi\in\mathbb{R} : 2^{m_j k}\delta\leq |\xi|\leq 2^{m_j k+1}\delta\}$ with $\delta\geq 1$.  Then 
for any $(1/p_1, \dots, 1/p_n, 1/r)\in\mathbb P_n$, 
there exists $c>0$ 
such that 
\begin{equation}\label{smo}
\big\| T_k(f_1, \dots, f_n)\big\|_r\lesssim \delta^{-c} \prod_{j=1}^n \|f_j\|_{p_j}\,. 
\end{equation}
\end{proposition}

\begin{proof}
Recall that 
\begin{equation*}
T_k(f_1, \dots, f_n)(x)=\int f_1(x-t^{m_1}) \cdots f_n(x-t^{m_n}) 2^{k}\rho(2^{k}t)\,\mathrm{d}t\,. 
\end{equation*}
Then \eqref{smo} is a consequence of standard interpolation between the inequalities in Theorem~\ref{thm1/2} and Proposition \ref{prop-0}.
\end{proof}

The kernel $1/t$ can be essentially written as  
$$
 \sum_{k\in \mathbb Z} 2^{k}\rho(2^{k} t)\,,
$$
where $\rho$ is a smooth odd function supported in $ \{t\in \mathbb R: 1/2<|t|<2\}$.
Hence,   we have 
\begin{equation*}
 H_{\Gamma_n}(f_1,\dots, f_n) = \sum_{k\in \mathbb Z} T_k(f_1, \dots, f_n)\,,
\end{equation*}
where $T_k$ is defined as in \eqref{defT_k}.

We shall further decompose the operator $T_k$ in frequency space. To this end, let $\Phi$ be a Schwartz function whose Fourier transform is supported in $\{\xi\in\mathbb R: 1/2<|\xi|<2\}$ and satisfies
$$
 \sum_{\delta\in 2^{\mathbb Z}} \wh\Phi \big( \frac{\xi}{\delta}\big) =1\,, \,\, {\forall }\,\xi\in\mathbb R\backslash\{0\}\,.
$$
Thus, the function $\Phi$ defines a smooth dyadic partition of unity on the real line.  
Given any integer $k$, for each function $f_j$, $j=1,\dots,n$, we decompose $f_j$ into dyadic frequency pieces: 
\begin{equation*}
 f_j =\sum_{\delta \in 2^{\mathbb Z}} f_{j, k, \delta}\,,
\end{equation*}
where  $f_{j, k, \delta}$ is defined by
\begin{equation}\label{s4-1}
\wh{f_{j, k, \delta}}(\xi) = \wh{f_j}(\xi) \wh\Phi\big( \frac{\xi}{2^{m_jk}\delta}\big) \,.
\end{equation}
Then we represent the multilinear Hilbert transform along $(t^{m_1}, \dots, t^{m_n})$ as 
$$
T(f_1, \dots,  f_n)=\sum_k \sum_{\delta_1, \dots, \delta_n\in 2^{\mathbb Z}} T_{k, \delta_1, \dots,  \delta_n}(f_1, \dots,  f_n)\,,
$$
where 
\begin{equation}\label{defT_delta}
T_{k, \delta_1, \dots, \delta_n}(f_1, \dots, f_n)(x):= 
\int f_{1, k, \delta_1}(x-t^{m_1}) \cdots f_{n,  k, \delta_n}(x-t^{m_n}) 2^k \rho(2^k t) \,\mathrm{d}t\,.
\end{equation}

First, notice that the multilinear operator $ T_{k, \delta_1, \dots, \delta_n}(f_1, \dots, f_n)(x)$ can also be represented in terms of the Fourier transforms of $f_j$ as  
$$
\iint\wh{f_{1, k, \delta_1}}(\xi_1)\cdots \wh{f_{n, k, \delta_n}}(\xi_n) e^{ix(\xi_1+\dots+\xi_n)}
  {\mathfrak{m}}_k({\xi_1, \dots, \xi_n}) \,\mathrm{d}\xi_1 \cdots \mathrm{d}\xi_n\,,
 $$
 where the multiplier $\mathfrak{m}_k$ is given explicitly by 
 $$
  {\mathfrak{m}}_k({\xi_1, \dots, \xi_n})= 
  \int e^{-i2^{-m_1 k}\xi_1t^{m_1}}\cdots e^{-i2^{-m_n k}\xi_n t^{m_n}}\rho(t)\,\mathrm{d}t\,.
 $$
 A key observation is that, whenever $\xi_j\in {\rm supp} \wh{ f_{j, k, \delta_j}}$, one has 
 \begin{equation*}
 2^{-m_j k}|\xi_j |\sim \delta_j\,.  
 \end{equation*}

We partition all possible dyadic tuples $(\delta_1, \dots, \delta_n)$ into  cases
 according to the number of indices $j$ for which $\delta_j>1$.  More precisely,  for $J\in \{0, 1, \dots, n\}$,  define 
 $$
 \mathcal D_J := \bigg\{ (\delta_1, \dots, \delta_n)\in 2^{\mathbb Z}\times \dots\times 2^{\mathbb Z}:  \#\big\{j\in\{ 1, \dots, n\}:\delta_j>1 \big\}=J\bigg\}\,.
 $$
 Then 
 \begin{equation*}
T(f_1, \dots,  f_n)= \sum_{J=0}^n\sum_k T_{k, J}(f_1, \dots, f_n)\,,
\end{equation*}
where
\begin{equation}\label{defTkl}
T_{k, J}(f_1, \dots, f_n) =  \sum_{(\delta_1, \dots, \delta_n)\in \mathcal D_J} T_{k, \delta_1, \dots,  \delta_n}(f_1, \dots,  f_n)\,.
\end{equation}

The main theorem, Theorem \ref{thm}, can be reduced to the following proposition.

\begin{proposition}\label{ell}
Let $T_{k, J}$ be defined as in \eqref{defTkl}. 
For any $(1/p_1, \dots, 1/p_n, 1/r)\in \mathbb P_n$ and any $J\in \{0, 1, \dots, n\}$,
 \begin{equation*}
 \Big\| \sum_k T_{k, J} (f_1, \dots,  f_n)\Big\|_{r}\leq C \|f_1\|_{p_1}\cdots \|f_n\|_{p_n}\,,
 \end{equation*}
 where  $C$ is independent of $f_1, \dots, f_n$.   
 \end{proposition}

The rest of this section is devoted to proving Proposition \ref{ell}.   The extremal case $J=n$ is the most favorable, as it follows directly from the Sobolev smoothing inequality in Proposition \ref{prop1}, together with the uniform estimates for the paraproducts in Theorem \ref{para}. The other extremal case, $J=0$, does not require Proposition \ref{prop1} and can be handled using only Theorem \ref{para}. The case $J=1$ can be reduced to the case $J=0$. For the intermediate cases $2\leq J\leq n-1$, the argument combines the paraproduct estimates with the Sobolev smoothing inequality, allowing these cases to be treated in a similar manner.

\subsection{The favorable extremal case: \texorpdfstring{$J=n$}{J=n}}\label{Jisn}

It follows from the definition of the Fourier multiplier ${\mathfrak m}_k$ that
$$
\sum_k\sum_{(\delta_1, \dots, \delta_n)\in\mathcal D_n}T_{k, \delta_1, \dots, \delta_n}(f_1, \dots, f_n)
= \sum_{\delta_1, \dots, \delta_n>1}\Pi_{\delta_1, \dots, \delta_n}(f_1, \dots, f_n)\,.
$$
Applying Theorems \ref{de-K} and  \ref{para}, we see that, for any $\e>0$, 
\begin{equation}\label{e}
 \big\|\Pi_{\delta_1, \dots, \delta_n}(f_1, \dots, f_n) \big\|_r\lesssim  \big( \prod_{j=1}^n \delta_j\big)^\e  \prod_{j=1}^n\big\|f_j\|_{p_j}\,,
\end{equation}
for all $ p_1, \dots, p_n>1$ and $r>1/2$. Employing the Sobolev smoothing inequality in Proposition \ref{prop1} for the tuple 
$\big(1/2,1/2, \e, \dots, \e, 1+(n-2)\e\big)\in\mathbb P_n$ and then applying interpolation, 
we obtain a positive constant $c$ such that
\begin{equation}\label{c}
\big\|\Pi_{\delta_1, \dots, \delta_n}(f_1, \dots, f_n) \big\|_r\lesssim  \big( \prod_{j=1}^n \delta_j\big)^{-c}  \prod_{j=1}^n\big\|f_j\|_{p_j}\,,
\end{equation}
for $(1/p_1, \dots, 1/p_n, 1/r)\in\mathbb P_n $.  
Interpolating between \eqref{e} and \eqref{c}, there exists a positive constant $c'$ such that
\begin{equation*}
\big\|\Pi_{\delta_1,\dots,\delta_n}(f_1,\dots,f_n)\big\|_r
\lesssim
\bigg(\prod_{j=1}^n\delta_j\bigg)^{-c'}
\prod_{j=1}^n\|f_j\|_{p_j},
\end{equation*}
for all $p_1,\dots,p_n>1$ and $r>1/2$. Summing over all possible values of $\delta_1,\dots,\delta_n>1$ then proves Proposition \ref{ell} for $J=n$.

\subsection{The intermediate cases: \texorpdfstring{$2\leq J\leq n-1$}{2 <= J <= n-1}}\label{J>2} 
For any $(\delta_1, \dots, \delta_n)\in\mathcal D_J$, 
there is a permutation $(i_1, i_2, \dots, i_n )$ of $(1, \dots, n)$ such that 
\begin{equation*}
 \min\big\{\delta_{i_1}, \dots, \delta_{i_J}\big\} >1 \,\,{\rm and} \,\,  \max\big\{\delta_{i_{J+1}}, \dots, \delta_{i_n}\big\} \leq 1\,. 
\end{equation*}
Without loss of generality, we may assume that 
$$
 \min\big\{\delta_{1}, \dots, \delta_{J}\big\} >1 \,\,{\rm and} \,\,  \max\big\{\delta_{{J+1}}, \dots, \delta_{n}\big\} \leq 1\,. 
 $$

For $J+1\leq j \leq n$,  we expand the exponential function $e^{-i2^{-m_{j}k}\xi_{j} t^{m_{j}} }$ into its
Taylor series to obtain 
\begin{equation*}
e^{-i2^{-m_{j}k}\xi_{j} t^{m_{j}} }  = \sum_{h=0}^\infty \frac{(-i)^h}{h!} \big(  2^{-m_{j} k} \xi_{j} \big)^h  t^{m_{j} h}\,,
\end{equation*}
which behaves exactly like 
$$
 \sum_{h=0}^{\infty} \frac{(-i)^h}{h!} \delta_{j}^h t^{m_{j}h}\,,
$$
whenever $\xi_{j} $ lies in the support of $\wh{\Phi_{j, k, \delta_{j}}}$.
Consequently, the Fourier multiplier ${\mathfrak m}_k$ satisfies
\begin{equation*}
 {\mathfrak{m}}_k \sim \sum_{\substack{h_{J+1}, \dots, h_n\geq 0 }} \frac{(-i)^{h_{J+1}+\dots+ h_n}}{h_{J+1}!\cdots h_n!}  {\mathfrak m}_{k, h_{J+1}, \dots, h_n}\,, 
\end{equation*} 
  where 
$$
{\mathfrak m}_{k, h_{J+1}, \dots, h_n} (\xi_1, \dots, \xi_n)=
 \int  \prod_{1\leq j \leq J} e^{-i2^{-m_{j}k }\xi_{j} t^{m_{j}}}   \prod_{j=J+1}^n  \delta_{j}^{h_j} t^{m_{j}h_j} \rho(t)\,\mathrm{d}t\,. 
$$
Due to the rapid decay of the factors $1/h_j!$ for all $J+1\leq j\leq n$, it suffices to consider ${\mathfrak m}_{k,h_{J+1},\dots,h_n}$ with $(h_{J+1},\dots,h_n)$ fixed. Let
$$
 \mathbb J = \big\{ j\in\{J+1, \dots, n\}: h_j =0\big\}  \,\,{\rm and}\,\, \mathbb J^c= \{J+1, \dots, n\}\backslash \mathbb J
 \,.
$$
For $j\in\mathbb J$, summing over the dyadic numbers $\delta_j$ gives
\begin{equation*}
\sum_{\delta_j\leq1}\widehat{\Phi}\bigg(\frac{\xi_j}{2^{m_jk}\delta_j}\bigg)
=\widehat{\Psi}\bigg(\frac{\xi_j}{2^{m_jk}}\bigg),
\end{equation*}
where $\widehat{\Psi}$ is a smooth function supported in $[-2,2]$. 
Hence we have, for $j\in\mathbb J$,  
\begin{equation*}
 \sum_{\delta_j\leq 1} f_{j, k, \delta_j} = f_{j, k, 1, \Psi}\,. 
\end{equation*}

Suppose that the set $\mathbb J^c$ contains exactly the indices $j_1,\dots,j_\ell\in\{J+1,\dots,n\}$.
Then we can represent
\begin{equation*}
 \sum_kT_{k, J}(f_1, \dots, f_n) = \sum_{\substack{h_{J+1}, \dots, h_n\geq 0 }} \frac{(-i)^{h_{J+1}+\dots+ h_n}}{h_{J+1}!\cdots h_n!}  
  T_{J, h_{J+1}, \dots, h_n}(f_1, \dots, f_n)\,,
 \end{equation*}
where $T_{J, h_{J+1}, \dots, h_n}(f_1, \dots, f_n)$ behaves like
$$
 \sum_{\substack{\delta_1, \dots, \delta_J>1\\ \delta_{j_1}, \dots,\delta_{j_\ell}\leq 1 }}
 \delta_{j_1}^{h_{j_1}}\cdots\delta_{j_\ell}^{h_{j_\ell}} \Pi_{\delta_1, \dots, \delta_J}(f_1, \dots, f_n)\,.
$$
As in Subsection \ref{Jisn}, it follows from Theorems \ref{de-K} and \ref{para}, together with Proposition \ref{prop1}, that there exists a constant $c'>0$ such that
\begin{equation*}
\big\|\Pi_{\delta_1,\dots,\delta_J}(f_1,\dots,f_n)\big\|_r
\lesssim
\bigg(\prod_{j=1}^J \delta_j\bigg)^{-c'}
\prod_{j=1}^n\|f_j\|_{p_j},
\end{equation*}
for all $p_1, \dots, p_n>1$ and $r>r_J$ with $1/r=1/p_1+\dots+ 1/p_n$.  
Summing over all possible values of $\delta_j$, we obtain the $(p_1,\dots,p_n,r)$ estimates for $r>r_J$, which clearly imply Proposition \ref{ell} for $2\leq J\leq n-1$.

\subsection{The degenerate extremal case:  \texorpdfstring{$J=0$}{J=0}}
Repeating the argument from Subsection \ref{J>2}, we obtain the following representation:
\begin{equation}\label{Tk0}
 \sum_kT_{k, 0}(f_1, \dots, f_n) = \sum_{\substack{h_1, \dots, h_n\geq 0 \\
  (h_1, \dots, h_n)\neq (0, \dots, 0)}} \frac{(-i)^{h_{1}+\dots+ h_n}}{h_{1}!\cdots h_n!}  
  T_{0, h_{1}, \dots, h_n}(f_1, \dots, f_n)\,,
 \end{equation}
where $T_{ 0, h_{1}, \dots, h_n}(f_1, \dots, f_n)$ behaves like
$$
 \sum_{\substack{ \delta_{j_1}, \dots,\delta_{j_\ell}\leq 1 }}
 \delta_{j_1}^{h_{j_1}}\cdots\delta_{j_\ell}^{h_{j_\ell}} \Pi_{0}(f_1, \dots, f_n)\,.
$$
Here, we have used the cancellation property of $\rho$, namely, $\int \rho(t)\,\mathrm{d}t=0$. 
It is clear that there exists at least one $h_\ell\neq 0$ on the right side of \eqref{Tk0}. 
This yields that the paraproduct $\Pi_0$ is admissible, and  Proposition \ref{ell} follows immediately from Theorem \ref{para}.

\subsection{The reducible case: \texorpdfstring{$J=1$}{J=1}}

For any $( \delta_1, \dots, \delta_n)\in \mathcal D_1$,  
there is a permutation $ (i_1, \dots, i_n)$ of $(1, \dots, n)$ such that 
\begin{equation*}
 \delta_{i_1}>1 \,\, {\rm and}\,\,  \max\big\{ \delta_{i_2}, \dots, \delta_{i_n}\big\}\leq 1\,.
\end{equation*}
Without loss of generality, we may assume that 
$$
 \delta_{1}>1 \,\, {\rm and}\,\,  \max\big\{ \delta_{2}, \dots, \delta_{n}\big\}\leq 1\,.
 $$
 As we did in Subsection \ref{J>2}, we have that
 the Fourier multiplier ${\mathfrak m}_k$ satisfies
\begin{equation*}
 {\mathfrak{m}}_k \approx \sum_{\substack{h_{2}, \dots, h_n\geq 0 }} \frac{(-i)^{h_{2}+\dots+ h_n}}{h_{2}!\cdots h_n!}  {\mathfrak m}_{k, h_{2}, \dots, h_n}\,, 
\end{equation*} 
  where 
$$
{\mathfrak m}_{k, h_{2}, \dots, h_n} (\xi_1, \dots, \xi_n)=\prod_{j=2}^n  \delta_{j}^{h_j}
 \varphi_{h_2, \dots, h_n}\Big( \frac{2^{-m_1k}\xi_1}{\delta_1}\Big)
 $$
with
 $$
 \varphi_{h_2, \dots, h_n}(x)= \int   e^{-i\delta_1 x t^{m_{1}}}   \prod_{j=2}^n  t^{m_{j}h_j} \rho(t)\,\mathrm{d}t\,. 
$$
Recall that when $\xi_1$ lies in the support of $\wh{\Phi_{1,k, \delta_1}}$, 
 $$
 2^{-m_1k} |\xi_1| \sim \delta_1\,,
 $$
 which allows us to restrict $\varphi_{h_2, \dots, h_n}$ smoothly to the region $|x|\sim 1$.  It is easy to see that, for any $\ell\in\mathbb N\cup \{0\}$, 
\begin{equation*}
 \big| D^{\ell} \varphi_{h_2, \dots, h_n}(x)\big| \lesssim_{\ell, N, m_1, \dots, m_n}
2^{\sum_{j=2}^{n}m_j h_j}\Big(1+\sum_{j=2}^{n} h_j\Big)^{N+\ell}\delta_1^{-N},
 \quad |x|\sim 1\,,
 \end{equation*}
 for any $N\in\mathbb N$.  Using Fourier series, we get
 \begin{equation*}
  \varphi_{h_2, \dots, h_n}(x)\wh\Phi(x) =\sum_{m\in\mathbb Z}  C_{m} e^{-imx }\, ,
   \end{equation*}
 where the Fourier coefficient $C_m$ satisfies 
 $$
  \big| C_m \big|\lesssim_{N, m_1, \dots, m_n} 2^{\sum_{j=2}^{n}m_j h_j}\Big(1+\sum_{j=2}^{n}h_j\Big)^{2N}\frac{\delta_1^{-N}}{(1+|m|)^N}\,.
 $$

As in Subsection \ref{J>2},  we have 
\begin{equation*}
 \sum_kT_{k, 1}(f_1, \dots, f_n) =\sum_m  \sum_{\delta_1>1}\sum_{\substack{h_{2}, \dots, h_n\geq 0 }} \frac{(-i)^{h_{2}+\dots+ h_n}C_m}{h_{2}!\cdots h_n!}  
  T_{h_2, \dots, h_n, \delta_1}^{(m)}(f_1, \dots, f_n)\,,
 \end{equation*}
where $T_{h_2, \dots, h_n,\delta_1}^{(m)}(f_1, \dots, f_n)$ behaves like
$$
 \sum_{\delta_{j_1}, \dots,\delta_{j_\ell}\leq 1}
 \delta_{j_1}^{h_{j_1}}\cdots\delta_{j_\ell}^{h_{j_\ell}} \Pi_{\delta_1}^{(m)}(f_1, \dots, f_n)\,.
$$
Here $\Pi_{\delta_1}^{(m)}(f_1,\dots,f_n)$ is given explicitly by
\begin{equation*}
 \Pi_{\delta_1}^{(m)}(f_1,\dots,f_n)
 =
 \sum_k f_{1,k,\delta_1,m}\prod_{j=2}^n f_{j,k,\delta_j,\Psi},
\end{equation*}
where
\[
f_{1,k,\delta_1,m}(x)
=
f_1*\Phi_{1,k,\delta_1}
\big(x-m2^{-m_1k}\delta_1^{-1}\big).
\]
Due to the rapid decay provided by $C_m$, the shift $m2^{-m_1k}\delta_1^{-1}$ causes no significant difficulty. Thus, we may assume $m=0$, since the case of general $m$ can be treated in exactly the same way and yields the same estimates. We are therefore reduced to the same type of admissible paraproduct as in the case $J=0$. This allows us to conclude Proposition \ref{ell} for $J=1$.\\

Therefore, the proof of Proposition \ref{ell} is complete, and consequently the proof of our main theorem, Theorem \ref{thm}, is also complete.\\ 

The main idea of the proof is to reduce the original problem to a family of uniform paraproduct estimates. This reduction is made possible by the Sobolev smoothing inequality, which allows us to control the relevant frequency-localized terms uniformly across the different parameter regimes. Once the problem is reduced to these uniform estimates, the desired bounds follow from the established paraproduct theory, whose proof will be provided in Sections \ref{Banach} and \ref{Pf-para}.


\section{Variants of Littlewood--Paley and maximal function theory}\label{VLP}

Recall that $\Phi$ is a Schwartz function whose Fourier transform is supported in $\{\xi\in \mathbb R: 1/2\leq |\xi|\leq 2 \}$ and 
that $\Phi_{j, k, \delta}$ is defined as in \eqref{s1-2}.

\begin{theorem}\label{LP}
Let $\beta \in\mathbb R$. Then for any $1<p<\infty$, 
\begin{equation*}
\bigg\|\bigg( \sum_k \big| f*\Phi_{j, k, \delta}\big(\cdot -2^{-m_jk}\delta^{-1} \beta \big)\big|^2\bigg)^{1/2} \bigg\|_{p}
\lesssim  \log\big( 10+|\beta|\big) \|f\|_p
\end{equation*}
holds for all $f\in L^p$.  Here the implicit constant is independent of $\delta$, $\beta$,  and $m_j\in\mathbb N$. 
\end{theorem}

\begin{proof}
By the $L^2$-Littlewood--Paley theory, we get
\begin{equation*}
\bigg\|\bigg( \sum_k \big| f*\Phi_{j, k, \delta}\big(\cdot -2^{-m_jk}\delta^{-1} \beta \big)\big|^2\bigg)^{1/2} \bigg\|_{2}
\lesssim  \|f\|_2\,,
\end{equation*}
which yields Theorem \ref{LP} for $p=2$, with an operator norm that is uniform in $\beta$, $\delta$, and $m_j$.
The square function in the theorem can be viewed as a vector-valued Calder\'on--Zygmund operator, whose convolution-type kernel
is given by
\begin{equation}\label{ker}
 \vec{K}_\beta(x) = \big\{ \Phi_{j, k, \delta}(x-2^{-m_jk}\delta^{-1}\beta)\big\}_{k}:= \big\{ K_{k, \beta}(x)\big\}_k\,.
\end{equation}
To complete the proof, by the standard Calder\'on--Zygmund theory,
it suffices to verify the following H\"ormander condition: 
\begin{equation}\label{Hor2}
\int_{|x|> 2|y|}  \sum_k \big| K_{k, \beta}(x-y)- K_{k, \beta}(x)\big| \,\mathrm{d}x \lesssim \log(10+|\beta|)\,.
\end{equation}
We may assume that $|\beta|\geq 10$, since the estimate for $|\beta|<10$ follows readily. Without loss of generality, we may further assume that $\beta\geq10$. We split the sum on the left-hand side of \eqref{Hor2} into three parts:
$$
 S_1: = \int_{|x|> 2|y|}   \sum_{\substack{k\\ \frac{1}{\beta^{100}}\leq |2^{m_jk}\delta y|\leq 8\beta}}  \big| K_{k, \beta}(x-y)- K_{k, \beta}(x)\big| \,\mathrm{d}x\,,
 $$
 $$
 S_2: = \int_{|x|> 2|y|}   \sum_{\substack{k\\  |2^{m_jk}\delta y| < \frac{1}{\beta^{100}}}}  \big| K_{k, \beta}(x-y)- K_{k, \beta}(x)\big| \,\mathrm{d}x\,,
 $$ 
 and 
 $$
 S_3: = \int_{|x|> 2|y|}   \sum_{\substack{k\\  |2^{m_jk}\delta y| > 8\beta  }}  \big| K_{k, \beta}(x-y)- K_{k, \beta}(x)\big| \,\mathrm{d}x\,.
 $$  
 There are only $O(\log \beta)$ values of $k$ in the sum defining $S_1$. We simply apply the triangle inequality and integrate the absolute values of the two terms in the difference to obtain
  \begin{equation}\label{S1}
 S_1\lesssim \log \beta\,, 
 \end{equation}  
 since $\sup_k\|K_{k, \beta}\|_1\leq C$.  
 
 To estimate $S_2$, we apply the mean value theorem to obtain
 $$
   \big| K_{k, \beta}(x-y)- K_{k, \beta}(x)\big| \lesssim \frac{ 2^{2m_jk}\delta^2 |y| }{(1+  \beta^{-1}2^{m_jk}\delta|x| )^N}\,,
 $$ 
 for any $N\in\mathbb N$, provided that $|x|>2|y|$. Thus, we have 
 \begin{equation}\label{S2}
  S_2\lesssim \int_{|x|> 2|y|} \sum_{\substack{k\\  |2^{m_jk}\delta y| < \frac{1}{\beta^{100}}}}   \frac{ 2^{2m_jk}\delta^2 |y| }{(1+  \beta^{-1}2^{m_jk}\delta|x| )^N} \,\mathrm{d}x
  \lesssim C\,.
 \end{equation}
 
  For the values of $k$ appearing in the sum defining $S_3$, we have
   $|2^{m_jk}\delta x| > 16 \beta$ whenever $|x|>2|y|$.    
 By the rapid decay of the Schwartz function $\Phi$,  we obtain
 \begin{equation*}
  \big| K_{k, \beta}(x-y)- K_{k, \beta}(x)\big| \lesssim \frac{2^{m_jk}\delta}{(1+ 2^{m_jk}\delta|x|)^N}\,, 
  \end{equation*}
 for any $N\in\mathbb N$. Hence, we have 
\begin{equation}\label{S3}
  S_3\lesssim \int_{|x|> 2|y|}   \sum_{\substack{k\\  |2^{m_jk}\delta y| > 8\beta  }}  \frac{2^{m_jk}\delta}{(1+ 2^{m_jk}\delta|x|)^N} \,\mathrm{d}x
  \lesssim C\,.
 \end{equation}
 
 Combining \eqref{S1}, \eqref{S2}, and \eqref{S3}, we obtain the desired H\"ormander-type estimate \eqref{Hor2}, which completes the proof.
 \end{proof}

 \begin{corollary}\label{corLP}
For any $\beta\in\mathbb R$, define the maximal operator
\begin{equation*}
T^*f(x)=\sup_k\big|
f*\Phi_{j,k,\delta}
\big(x-2^{-m_jk}\delta^{-1}\beta\big)
\big|.
\end{equation*}
Then, for any $1<p<\infty$,
\begin{equation*}
\|T^*f\|_p
\lesssim
\log(10+|\beta|)\|f\|_p.
\end{equation*}
Here, the implicit constant is independent of $\delta$, $\beta$, and $m_j\in\mathbb N$.
\end{corollary}

\begin{proof}
By the pointwise inequality
$\sup_k |a_k|\leq \big(\sum_k |a_k|^2\big)^{1/2}$, we have
\[
|T^*f(x)|\leq \bigg(\sum_k\big|f*\Phi_{j,k,\delta}\big(x-2^{-m_jk}\delta^{-1}\beta\big)\big|^2\bigg)^{1/2}.
\]
Therefore, by Theorem \ref{LP},
\[
\|T^*f\|_p\lesssim \log(10+|\beta|)\|f\|_p\,,
\]
as desired.
\end{proof}

 
\section{Banach-space case: \texorpdfstring{$r\geq 1$}{r>=1}}\label{Banach}

It turns out that the Banach-space case $r\geq1$ is simpler than the quasi-Banach case $0<r<1$. We present the proof of the Banach-space case in this section.

\begin{proposition}\label{Lr>1}
Let $n\geq 2$ and $J\in\{2, \dots, n\}$. Let $\Pi_{\delta_1, \dots, \delta_J}$ be an admissible paraproduct.
Then there exists a constant $C$, independent of $\delta_1, \dots, \delta_n$ and $n_1, \dots, n_J$, such that
\begin{equation}\label{bana-J>2}
\big\| \Pi_{\delta_1, \dots, \delta_J}  (f_1, \dots, f_n) \big\|_r \leq C \prod_{j=1}^J n_j  \prod_{j=1}^n \big\|f_j\big\|_{p_j}\,,
\end{equation}
 for all $p_1, \dots, p_n>1$ and $r\geq 1$, provided that $(p_1, \dots, p_n, r)$ obeys the H\"older relation 
\eqref{homo-r}.
\end{proposition}

\begin{proof}
Since $J\geq 2$, we can choose  two distinct Littlewood--Paley-type indices $i_1, i_2\in \{1, \dots, J\}$ and write 
the paraproduct $ \Pi_{\delta_1, \dots, \delta_J}$ as
\begin{equation*}
\sum_{k\in\mathbb K} \int \rho^*(t) \!\!\!\prod_{j\in\{i_1, i_2\}}f_{j, k, \delta_j, n_j, t}(x)
\!\!\! \prod_{\substack{1\leq j\leq J \\ j\neq i_1, i_2}} f_{j, k, \delta_j, n_j, t}(x) \,\mathrm{d}t   \!\!\!\prod_{J+1\leq j\leq n}
 f_{j, k, \delta_j, \Psi}(x)\,.
\end{equation*}
Inserting the absolute values and then applying H\"older's inequality, we dominate $|\Pi_{\delta_1, \dots, \delta_J}|$
by 
$$
 \int\!\!\big| \rho^*(t)\big|  \!\! \!\!\!\prod_{j\in\{\!i_1, i_2\!\}}\!\!  \!\!\bigg(\!\!\sum_k\big| f_{j, k, \delta_j, n_j, t}(x)\big|^2\!\bigg)^{\!\frac{1}{2}} \!\!\!\!\!  \prod_{\substack{1\leq j\leq J \\ j\neq i_1, i_2}} \!\!\sup_k\big| f_{j, k, \delta_j, n_j, t}(x) \big|\mathrm{d}t 
  \!\!\!\!\prod_{j=J+1}^n \!\!\!\sup_k\big|  f_{j, k, \delta_j, \Psi}(x)\big|\,.
 $$
For $r\geq 1$, we use Minkowski's inequality to get 
\begin{eqnarray*}
  & & \big\| \Pi_{\delta_1, \dots, \delta_J}(f_1, \dots, f_n)\big\|_r\\
  &\leq &\!\! \!\!  \int\!\!\big| \rho^*(t)\big|  \bigg\|\!\!\prod_{j\in\{\!i_1, i_2\!\}}\!\!  \!\!\bigg(\!\!\sum_k\big| f_{j, k, \delta_j, n_j, t}\big|^2\!\bigg)^{\!\frac{1}{2}} \!\!\!\!\!  \prod_{\substack{1\leq j\leq J \\ j\neq i_1, i_2}} \!\!\sup_k\big| f_{j, k, \delta_j, n_j, t} \big|
  \!\!\!\!\prod_{j=J+1}^n \!\!\!\sup_k\big|  f_{j, k, \delta_j, \Psi}\big|\bigg\|_r\!\! \mathrm{d}t\\
 &\leq&  \prod_{j=1}^J n_j \int|\rho^*(t)| \,\mathrm{d}t\prod_{j=1}^n \|f_j\|_{p_j}\,.
  \end{eqnarray*}
In the last step, we applied H\"older's inequality with exponents $(p_1/r,\dots,p_n/r)$, together with Theorem \ref{LP} and Corollary \ref{corLP}. This completes the proof.  
\end{proof}

\begin{lemma}\label{lemP-02}
Suppose that there are two distinct Littlewood--Paley-type indices in the paraproduct $\Pi_0$ defined as in \eqref{def-Pi}.
 Then there exists a constant $C$, independent of $\delta_1, \dots, \delta_n$,  such that
\begin{equation*}
\big\| \Pi_{0}  (f_1, \dots, f_n) \big\|_r \leq C  \prod_{j=1}^n \big\|f_j\big\|_{p_j}\,,
\end{equation*}
 for all $p_1, \dots, p_n>1$ and $r\geq 1$, provided that $(p_1, \dots, p_n, r)$ obeys the H\"older relation 
\eqref{homo-r}.
\end{lemma}

\begin{proof}
This can be done exactly in the same manner as in the proof of Proposition \ref{Lr>1}. We omit the details.
\end{proof}

\begin{lemma}\label{lemPi0}
Let $n \geq 2$, and let $\Pi_0$ be a paraproduct defined as in \eqref{def-Pi}. Suppose that there is a Littlewood--Paley-type index $\ell \in \{1,\dots,n\}$ associated with the $\ell$-th function in the paraproduct $\Pi_0$. Then
\begin{equation}\label{r-geq-1}
\bigl\|\Pi_0(f_1,\dots,f_n)\bigr\|_{r}
\lesssim
\prod_{j=1}^n \|f_j\|_{p_j},
\end{equation}
for all $p_1,\dots,p_n>1$ and $r\geq 1$ satisfying
\[
\frac{1}{r}=\frac{1}{p_1}+\dots+\frac{1}{p_n}.
\]
\end{lemma}

\begin{proof}
We proceed by induction on $n$. The case $n=2$ was proved in \cite{Li2} and serves as the base case of the induction. Assume that Lemma \ref{lemPi0} holds for all integers $\leq n$. We aim to prove the lemma for $n+1$.

Let $\Pi_0(f_1,\dots,f_{n+1})$ be a paraproduct in which the index $\ell\in\{1,\dots,n+1\}$ is of Littlewood--Paley type. We may further assume that $\ell$ is the only LP-type index; otherwise, the conclusion follows immediately from Lemma \ref{lemP-02}.

We write
\[
\Pi_0(f_1,\dots,f_{n+1})
=
\sum_{k\in\mathbb K}
f_{\ell,k,\delta_\ell}
\prod_{j\neq \ell} f_{j,k,\delta_j,\Psi}.
\]

By Theorem \ref{de-K}, we may assume that $\mathbb K$ is a non-degenerate well-distributed set of order $K$, where $K$ is sufficiently large, since the degenerate case is trivial.

Suppose first that, for every $k\in\mathbb K$, the Fourier support of $f_{\ell,k,\delta_\ell}$ has scale much larger than that of any other factor $f_{j,k,\delta_j,\Psi}$ in the paraproduct, namely, $2^{m_jk}\delta_j\preceq_K2^{m_\ell k}\delta_\ell$ for every $j\neq\ell$ and every $k\in\mathbb K$.  In this case, the key observation is that the collection of $\{f_{\ell,k,\delta_\ell}
\prod_{j\neq \ell} f_{j,k,\delta_j,\Psi}\}_k$ is of Littlewood--Paley type. Hence we can apply Lemma \ref{RSFE} to obtain
\begin{equation}\label{rev-sq}
\big\| \Pi_0(f_1,\dots,f_{n+1})\big\|_r\lesssim \bigg\|\bigg( \sum_k \big| f_{\ell,k,\delta_\ell}
\prod_{j\neq \ell} f_{j,k,\delta_j,\Psi}\big|^2\bigg)^{1/2} \bigg\|_r\,.
\end{equation}
Since each function $|f_{j,k,\delta_j,\Psi}|$ is bounded by $C M f_j$, where $M$ denotes the maximal operator, \eqref{rev-sq} yields the following estimate:
$$
\big\| \Pi_0(f_1,\dots,f_{n+1})\big\|_r\lesssim \bigg\|\bigg( \sum_k \big| f_{\ell,k,\delta_\ell}
\big|^2\bigg)^{1/2} \prod_{j\neq \ell} Mf_{j} \bigg\|_r\,.
$$
By H\"older's inequality, this is bounded by 
$$
\bigg\| \bigg( \sum_k\big| f_{\ell, k, \delta_\ell}\big|^2\bigg)^{1/2}\bigg\|_{p_\ell}
\!\!\prod_{\substack{j\neq \ell\\j\in\{1, \dots, n+1\}}}\!\!\!\!\!\! \big\| Mf_j\big\|_{p_j} \,.
$$
Thus,  the desired $(p_1,\dots,p_{n+1},r)$-estimate now follows by combining this estimate with the Littlewood--Paley and maximal function theories. \\

Otherwise, since $\mathbb K$ is non-degenerate and well-distributed, there exists an index $\ell'\in \{1, \dots, n+1\}\backslash\{\ell\}$ such that, for the relevant  $k\in\mathbb K$, the Fourier support of
$f_{\ell',k,\delta_{\ell'},\Psi}$ has scale  much larger than that of $f_{\ell, k, \delta_\ell}$. Let $\mathbb L$ denote the collection of all such indices $\ell'$, and set 
$$
\mathbb L_\ell^c=\big\{j\in \{1, \dots, n+1\}: j\neq\ell \,\, {\rm and}\,\, j\notin \mathbb L \big\}\,.
$$
Then we have the following representation:
\begin{equation*}
\Pi_0(f_1,\dots,f_{n+1}) = \sum_{k\in\mathbb K} \big(\prod_{j\in \mathbb L} f_{j, k, \delta_j, \Psi}\big) f_{\ell, k, \delta_\ell} \prod_{j\in \mathbb L_\ell^c} f_{j, k, \delta_j, \Psi}\,.
\end{equation*}

Write $ \mathbb L=\{i_1, \dots, i_l\}$ so that for any $q\in \{1, \dots, l-1 \}$,
\[
2^{m_{i_q}k}\delta_{i_q} \geq \max_{j\in\{i_{q+1},\dots,i_l\}}2^{m_jk}\delta_j\, ,
\]
that is, the scale of the Fourier support of $f_{i_q, k, \delta_{i_q}}$ is larger than those of $f_{j, k, \delta_{j}}$ for $j\in \{i_{q+1}, \dots, i_l\}$.
By the partition of unity generated by the function $\Phi$,  we can write
$$
f_{i_1, k, \delta_{i_1}, \Psi} = f_{i_1} - \sum_{\delta>{\delta_{i_1}}} f_{i_1, k, \delta}\,,
$$
which leads to 
$$
\Pi_0(f_1,\dots,f_{n+1}) =  \Pi_0^{(1)}(f_1,\dots,f_{n+1})-\Pi_0^{(2)}(f_1,\dots,f_{n+1}) \,,
$$
where 
\begin{equation*}
 \Pi_0^{(2)}(f_1,\dots,f_{n+1}) = \sum_{k}   \bigg( \sum_{\delta>\delta_{i_1}}f_{i_1, k, \delta}\bigg) \bigg(\prod_{\substack{j\in\mathbb L\\ j\neq i_1}}f_{j, k, \delta_j, \Psi} \bigg) f_{\ell, k, \delta_\ell} \prod_{j\in \mathbb L_\ell^c} f_{j, k, \delta_j, \Psi}
\end{equation*}
and
\begin{equation*}
    \Pi_0^{(1)}(f_1,\dots,f_{n+1}) = \Pi_0(f_1,\dots,f_{n+1})+ \Pi_0^{(2)}(f_1,\dots,f_{n+1})\,.
\end{equation*}
The desired estimates for $\Pi_0^{(1)}(f_1,\dots,f_{n+1})$ follow by applying  H\"older's inequality and the inductive hypothesis. \\

We now turn to the analysis of the most difficult term, $\Pi_0^{(2)}(f_1,\dots,f_{n+1})$. For $h\in \mathbb Z$ and $j\in \mathbb L$, 
we define $f_{j, h}$ by
\begin{equation*}
 \wh{f_{j, h}}(\xi) = \wh{f_j}(\xi) \wh\Phi\big( \frac{\xi}{2^h}\big)\,.
\end{equation*}
It is clear that, for $j\in\mathbb L$,
\begin{equation}\label{dec-j}
 \sum_{\delta>\delta_j}f_{j, k, \delta} = \sum_{h>m_j k+ \log_2\delta_j} f_{j, h}\,.
\end{equation}

Then we can write $\Pi_0^{(2)}(f_1,\dots,f_{n+1}) $ as 
$$
\sum_{h_1}f_{i_1, h_1} G_{h_1}\big(\{f_j\}_{j\in \mathbb L\backslash \{i_1\}}\big)\,,
$$
where 
$$
G_{h_1}\big(\{f_j\}_{j\in \mathbb L\backslash \{i_1\}}\big) :=            \!\!
\sum_{\substack{k\in\mathbb K\\ h_{1}>  m_{i_1}k+\log_2\delta_{i_1}
}} \bigg(\prod_{\substack{j\in\mathbb L\\ j\neq i_1}}f_{j, k, \delta_j, \Psi} \bigg)  f_{\ell, k, \delta_\ell} \prod_{j\in \mathbb L_\ell^c} f_{j, k, \delta_j, \Psi}\,.
$$
Notice that the $h_1$-summand, $f_{i_1, h_1}G_{h_1}$,   is of Littlewood--Paley type, because $\mathbb K$ is well-distributed. From  Lemma \ref{RSFE}, it follows that 
\begin{equation*}
    \big\| \Pi_0^{(2)}(f_1,\dots,f_{n+1})\big\|_r \lesssim \bigg\|  \big( \sum_{h_1} \big| f_{i_1, h_1}G_{h_1}\big|^2 \big)^{1/2}\bigg\|_r\,, 
\end{equation*}
which is bounded by 
\begin{equation*}
\bigg\| \big(\sum_{h_1}\big| f_{i_1, h_{1}} \big|^2\big)^{1/2}\bigg\|_{p_{i_1}} 
\bigg\| \sup_{h_1}\big| G_{h_1}\big|\bigg\|_{r_1} \lesssim \|f_{i_1}\|_{p_{i_1}} \bigg\|\sup_{h_1}\big| G_{h_1}\big|\bigg\|_{r_1} \,,
\end{equation*}
where $1/r_1=1/r-1/p_{i_1}$.  
We now write 
$$
f_{i_2, k, \delta_{i_2}, \Psi} = f_{i_2} - \sum_{\delta>{\delta_{i_2}}} f_{i_2, k, \delta}\,.
$$
Then we decompose $G_{h_1}$ into $G_{h_1, 1} - G_{h_1, 2}$, where
\begin{equation*}
G_{h_1, 1}= f_{i_2}\sum_{\substack{k\in\mathbb K\\ h_{1}> m_{i_1}k+\log_2\delta_{i_1}
}} \bigg(\prod_{\substack{j\in\mathbb L\\ j\neq i_1, i_2}}f_{j, k, \delta_j, \Psi} \bigg)  f_{\ell, k, \delta_\ell} \prod_{j\in \mathbb L_\ell^c} f_{j, k, \delta_j, \Psi}
\end{equation*}
and 
\begin{equation*}
G_{h_1, 2}= \sum_{\substack{k\in\mathbb K\\ h_{1}> m_{i_1}k+\log_2\delta_{i_1} 
}} 
 \bigg( \sum_{\delta>{\delta_{i_2}}} f_{i_2, k, \delta} \bigg)
\bigg(\prod_{\substack{j\in\mathbb L\\ j\neq i_1, i_2}}f_{j, k, \delta_j, \Psi} \bigg)  f_{\ell, k, \delta_\ell} \prod_{j\in \mathbb L_\ell^c} f_{j, k, \delta_j, \Psi}\,.
\end{equation*}
Applying H\"older's inequality to $G_{h_1, 1}$, we get 
\begin{equation*}
\bigg\|\sup_{h_1}\big| G_{h_1, 1}\big|\bigg\|_{r_1}\lesssim \big\|f_{i_2}\big\|_{p_{i_2}} \bigg\| \sup_{h_1}\bigg|G_{h_1}\big(\{f_j\}_{j\in \mathbb L\backslash \{i_1, i_2\}}\big)  \bigg| \bigg\|_{r_2}\,,
\end{equation*}
where $1/r_2=1/r_1-1/p_{i_2}$.  Using \eqref{dec-j} for $j=i_2$,  we see that 
$$
G_{h_1, 2} = \sum_{h_2} f_{i_2, h_2} \!\!\!\!\!\sum_{\substack{k\in\mathbb K\\ h_{1}>  m_{i_1}k+\log_2\delta_{i_1}\\ h_{2}>  m_{i_2}k+\log_2\delta_{i_2}
}} 
\bigg(\prod_{\substack{j\in\mathbb L\\ j\neq i_1, i_2}}f_{j, k, \delta_j, \Psi} \bigg)  f_{\ell, k, \delta_\ell} \prod_{j\in \mathbb L_\ell^c} f_{j, k, \delta_j, \Psi}\,.
$$
At this point, we partition the range of $h_2$ to clarify the dependence of the $k$-sum on $h_1$ and $h_2$.  In fact, we have 
$$
 G_{h_1, 2} = G_{h_1, 2}^{(1)} + G_{h_1, 2}^{(2)}\,, 
$$
where 
\begin{equation*}
    G_{h_1, 2}^{(1)}= \!\!\!\!\sum_{\substack{h_2\\ \frac{h_1-\log_2\delta_{i_1}}{m_{i_1}}\geq \frac{h_2-\log_2\delta_{i_2}}{m_{i_2}} } }\!\!\!\! f_{i_2, h_2} \!\!\!\!\!\sum_{\substack{k\in\mathbb K\\  h_{2}>  m_{i_2}k+\log_2\delta_{i_2}
}} \!\!\!\!
\bigg(\prod_{\substack{j\in\mathbb L\\ j\neq i_1, i_2}}f_{j, k, \delta_j, \Psi} \bigg)  f_{\ell, k, \delta_\ell}\!\! \prod_{j\in \mathbb L_\ell^c} f_{j, k, \delta_j, \Psi}
\end{equation*}
and 
\begin{equation*}
    G_{h_1, 2}^{(2)}= \!\!\!\!\sum_{\substack{h_2\\ \frac{h_1-\log_2\delta_{i_1}}{m_{i_1}}< \frac{h_2-\log_2\delta_{i_2}}{m_{i_2}} } }\!\!\!\! f_{i_2, h_2} \!\!\!\!\!\sum_{\substack{k\in\mathbb K\\ h_{1}>  m_{i_1}k+\log_2\delta_{i_1}
}} \!\!\!\!
\bigg(\prod_{\substack{j\in\mathbb L\\ j\neq i_1, i_2}}f_{j, k, \delta_j, \Psi} \bigg)  f_{\ell, k, \delta_\ell}\!\! \prod_{j\in \mathbb L_\ell^c} f_{j, k, \delta_j, \Psi}\,.
\end{equation*}

Applying Lemma \ref{MaxRS}, followed by H\"older's inequality and the Littlewood--Paley theory, to $G^{(1)}_{h_1,2}$, we arrive at
\begin{equation*}
\bigg\|\sup_{h_1}\big| G_{h_1, 2}^{(1)}\big|\bigg\|_{r_1}\lesssim \big\|f_{i_2}\big\|_{p_{i_2}} \bigg\| \sup_{h_2}\bigg|G_{h_2}\big(\{f_j\}_{j\in \mathbb L\backslash \{i_1, i_2\}}\big)  \bigg| \bigg\|_{r_2}\,,
\end{equation*}
where 
$$
G_{h_2}\big(\{f_j\}_{j\in \mathbb L\backslash \{i_1, i_2\}}\big) :=            \!\!
\sum_{\substack{k\in\mathbb K\\ h_{2}>  m_{i_2}k+\log_2\delta_{i_2}
}} \bigg(\prod_{\substack{j\in\mathbb L\\ j\neq i_1, i_2}}f_{j, k, \delta_j, \Psi} \bigg)  f_{\ell, k, \delta_\ell} \prod_{j\in \mathbb L_\ell^c} f_{j, k, \delta_j, \Psi}\,.
$$
In $G^{(2)}_{h_1,2}$, the $k$-sum is independent of $h_2$ and can therefore be taken outside the $h_2$-sum. Applying H\"older's inequality followed by Lemma \ref{MaxRS}, we obtain the following similar estimate:
\begin{equation*}
\bigg\|\sup_{h_1}\big| G_{h_1, 2}^{(2)}\big|\bigg\|_{r_1}\lesssim \big\|f_{i_2}\big\|_{p_{i_2}} \bigg\| \sup_{h_1}\bigg|G_{h_1}\big(\{f_j\}_{j\in \mathbb L\backslash \{i_1, i_2\}}\big)  \bigg| \bigg\|_{r_2}\,.
\end{equation*}

By iterating the above argument, we can successively remove all $f_j$, $j\in\mathbb L$, from the functions of $G_h$-type, thereby arriving at 
\begin{equation}\label{Pi0-G}
   \big\|  \Pi_0^{(2)}(f_1,\dots,f_{n+1})\big\|_r\lesssim \bigg( \prod_{j\in\mathbb L}\big\|f_{j}\big\|_{p_j}\bigg)\!\! \max_{1\leq j\leq l}
   \bigg\|\sup_{h_j}\bigg|\!\!\!\!\!\sum_{\substack{k\in\mathbb K\\ h_j> m_{i_j}k+\log_2\delta_{i_j} }} \!\!\!\!\!\!\!\!\!f_{\ell, k, \delta_\ell} \prod_{j\in \mathbb L_\ell^c} f_{j, k, \delta_j, \Psi} \bigg|\bigg\|_{p}\,, 
\end{equation}
where $1/p=1/r-\sum_{j\in\mathbb L}1/p_j$. Observe that, for every $k\in\mathbb K$, the Fourier support of $f_{\ell,k,\delta_\ell}$ has scale much larger than that of any other factor $f_{j,k,\delta_j,\Psi}$ with $j\in\mathbb L_\ell^c$ appearing in the paraproduct. By Lemma \ref{MaxRS}, the second factor involving $f_\ell$ on the right-hand side of (\ref{Pi0-G}) is controlled by
\begin{equation}\label{ell-G}
 \bigg\|\big(\sum_{k\in\mathbb K} \big|f_{\ell, k, \delta_\ell} \prod_{j\in \mathbb L_\ell^c} f_{j, k, \delta_j, \Psi} \big|^2\big)^{1/2}\bigg\|_{p}\,,
\end{equation}
which, by an application of H\"older's inequality,  the  Littlewood--Paley theory and the maximal function theory, is further dominated by
$$
\bigg\|\big(\sum_{k\in\mathbb K} \big|f_{\ell, k, \delta_\ell} \big|^2\big)^{1/2}\prod_{j\in \mathbb L_\ell^c} Mf_{j} \bigg\|_{p}\lesssim  \|f_\ell\|_{p_\ell}\prod_{j\in \mathbb L^c_\ell}\|f_j\|_{p_j}\,
$$
as desired.  Therefore, the proof is complete.
\end{proof}

Combining Proposition \ref{Lr>1}, Lemma \ref{lemP-02}, and Lemma \ref{lemPi0}, we obtain Theorem \ref{para} in the Banach-space case, namely,
\begin{proposition}
Theorem \ref{para} holds for $r\geq 1$.
\end{proposition}

\section{Proof of Theorem \ref{para}: quasi-Banach cases}\label{Pf-para}

In this section we consider the boundedness in the quasi-Banach spaces $L^r$. 

\subsection{The untranslated paraproducts}

The simplest untranslated paraproduct is the one for which there exists a Littlewood–Paley index $\ell$ such that, for every $k\in\mathbb K$, the Fourier support of $f_{\ell,k,\delta_\ell}$ has length much larger than that of any other factor $f_{j,k,\delta_j,\Psi}$ in the paraproduct. We call such a paraproduct {\it regular},  and the index $\ell$ is called a dominant index.

\begin{lemma}
Let $\Pi_0$ be a regular paraproduct.  Then 
\begin{equation}\label{reg}
 \big\|\Pi_0(f_1, \dots, f_n)\big\|_r\lesssim \prod_{j=1}^n\big\|f_j\big\|_{p_j}\,
\end{equation}
holds for all $p_1, \dots, p_n>1$ and $r>1/n$ with $1/r=\sum_{j=1}^n 1/p_j$. 
\end{lemma}

\begin{proof}
Without loss of generality, we assume that $1$ is the dominant index.  From Lemma \ref{RSFE}, it follows 
that 
\begin{equation}\label{eqn-r-para}
\big\|\Pi_0(f_1, \dots, f_n)\big\|_r\lesssim 
 \bigg\|\bigg(\sum_{k\in\mathbb K} \big|f_{1, k, \delta_1, \Psi} \prod_{j=2}^n f_{j, k, \delta_j, \Psi} \big|^2\bigg)^{1/2}\bigg\|_{r}\,.
\end{equation}
Since (\ref{eqn-r-para}) is of the same form as (\ref{ell-G}), the same argument as in the proof of Lemma \ref{lemPi0} yields the desired estimate.
\end{proof}

\begin{lemma}
Let $r=1/n+\sigma$, where $\sigma$ is a sufficiently small positive number. Let $\Pi_0$ be an admissible  paraproduct.  Then  
\begin{equation}\label{near1-n}
 \big\|\Pi_0(f_1, \dots, f_n)\big\|_r\lesssim \prod_{j=1}^n\big\|f_j\big\|_{p_j}\,
\end{equation}
holds for all $p_1, \dots, p_n>1$ and $r>1/n$ with $1/r=\sum_{j=1}^n 1/p_j$. 
\end{lemma}

\begin{proof}
Notice that when $r$ is sufficiently close to $1/n$, each $p_j$ is sufficiently close to $1$. Consequently,
$$
 1/r-\sum_{m=1}^{l}1/p_{i_m} = n-l +O(\sigma)\,.
$$
This observation makes the induction possible, allowing us to proceed by following the same argument as in the proof of Lemma \ref{lemPi0}. We are then reduced to estimating (\ref{Pi0-G}). This estimate was established in the proof of Lemma \ref{lemPi0}, using Lemma \ref{MaxRS}, H\"older's inequality, the Littlewood--Paley theory, and the maximal function estimates. Therefore, the proof is complete.
\end{proof}

\begin{proposition}
Let $\Pi_0$ be an admissible paraproduct. Then
\begin{equation*}
\bigl\|\Pi_0(f_1,\dots,f_n)\bigr\|_{r}
\lesssim
\prod_{j=1}^n \|f_j\|_{p_j}
\end{equation*}
holds for all $p_1,\dots,p_n>1$ and $r>1/n$ satisfying
\[
\frac{1}{r}=\frac{1}{p_1}+\dots+\frac{1}{p_n}.
\]
\end{proposition}

\begin{proof}
This is immediate from interpolation between \eqref{r-geq-1} and \eqref{near1-n}.
\end{proof}

\subsection{The translated paraproducts}

In this subsection, we use Lemma \ref{MaxRS} to establish the desired estimates for the $J$-translated paraproduct over the  range of $r>r_J$. We use $\Id_F$ to denote the indicator function of the measurable set $F$. To establish Theorem \ref{para}, we first prove the following restricted-type estimates.

\begin{proposition}
Let $r=r_J+\sigma$, where $\sigma$ is a sufficiently small positive number. 
\begin{itemize}
\item 
If $J=2$,  
then for any $\e>0$,
\begin{equation}\label{r_J-est}
 \big\| \Pi_{\delta_1, \dots, \delta_J} (\Id_{F_1}, \dots, \Id_{F_n})\big\|_{r} \lesssim 2^{\e n_J^*}
 \prod_{j=1}^{n} \big| F_j\big|^{1/p_j}
\end{equation}
holds for all $p_1, \dots, p_n>1$ and $r>r_J$ such that 
\begin{equation*}
\frac{1}{r}= \frac{1}{p_{1}} +\dots +\frac{1}{p_n}\,. 
\end{equation*}

\item If $J \geq 3$, let $i_1, i_2\in \{1, \dots , J\}$ be any two distinct indices. 
Then for any $\e>0$,
\begin{equation}\label{r_J-est>2}
 \big\| \Pi_{\delta_1, \dots, \delta_J} (\Id_{F_1}, \dots, \Id_{F_n})\big\|_{r} \lesssim 2^{\e n_J^*}\big|F_{i_1}\big|^{\frac{1}{p_{i_1}}}
 \big| F_{i_2}\big|^{\frac{1}{p_{i_2}}}  \!\!
 \prod_{j=J+1}^n \!\big| F_j\big|^{1/p_j}
\end{equation}
holds for all $p_{i_1}, p_{i_2}, p_{J+1}, \dots, p_n>1$ and $r>r_J$ such that 
$$
\frac{1}{r}= \frac{1}{p_{i_1}} +\frac{1}{p_{i_2}}+\sum_{j=J+1}^n\frac{1}{p_j}\,. 
$$
\end{itemize}
\end{proposition}

Define the $(n+1)$-linear form $\Lambda_{\delta_1, \dots, \delta_J}$ by
\begin{equation*}
\Lambda_{\delta_1, \dots, \delta_J}(f_1, \dots, f_{n+1})= \big\langle \Pi_{\delta_1, \dots, \delta_J}(f_1, \dots, f_n), f_{n+1} \big\rangle\,. 
\end{equation*}
Estimate \eqref{r_J-est} can be reduced to the following estimate: for any measurable set $F_{n+1}$ with $|F_{n+1}|=1$, 
there exists a subset $F'_{n+1}$ of $F_{n+1}$ with $|F_{n+1}'|\geq 1/2$ such that
\begin{equation}\label{r_Jest1}
\Lambda_{\delta_1, \dots, \delta_J}(\Id_{F_1}, \dots, \Id_{F_n}, \Id_{F'_{n+1}})
\lesssim 2^{\e n_J^*} \prod_{j=1}^{n} \big| F_j\big|^{1/p_j} \,.
\end{equation}
Similarly, estimate \eqref{r_J-est>2} can be reduced to the following estimate: for any measurable set $F_{n+1}$ with $|F_{n+1}|=1$, 
there exists a subset $F'_{n+1}$ of $F_{n+1}$ with $|F_{n+1}'|\geq 1/2$ such that
\begin{equation}\label{r_Jest1>2}
\Lambda_{\delta_1, \dots, \delta_J}(\Id_{F_1}, \dots, \Id_{F_n}, \Id_{F'_{n+1}})
\lesssim 2^{\e n_J^*}\!\big|F_{i_1}\big|^{1/p_{i_1}}\!
 \big| F_{i_2}\big|^{1/p_{i_2}}\!\!\!\prod_{j=J+1}^{n}\! \!\!\!\big| F_j\big|^{1/p_j} \,.
\end{equation}

\begin{definition}
Define the exceptional set $\Omega$ by 
\begin{equation*}
\Omega= \bigcup_{j=1}^{n+1} \big\{ x\in\mathbb R:  M\Id_{F_j}(x)> C_0 |F_j|\big\}\,,
\end{equation*}
where $Mf$ is the Hardy--Littlewood maximal function of $f$.
\end{definition}

We take $F_{n+1}'=F_{n+1}\backslash\Omega$. If $C_0$ is chosen sufficiently large, then we see that $|F_{n+1}'|\geq |F_{n+1}|/2$.
With the exceptional set removed, we can allow the range of 
$r$ to extend below $1$. The key observation is that, when $f_j=\Id_{F_j}$, we have 
\begin{equation*}
 \big| f_{j, k, \delta_j, \Psi}(x)\Id_{\Omega^c}(x) \big| \lesssim  \int\frac{2^{m_j k}\delta_j\Id_{F_j}(y)}{\big(1+2^{m_jk}\delta_j|x-y|\big)^{N}} \,\mathrm{d}y \,\Id_{\Omega^c}(x) \,,
\end{equation*}
which is bounded by 
$$M\Id_{F_j}(x)\Id_{\Omega^c}(x) \lesssim \min\{1, |F_j|\} \lesssim |F_j|^{1/p_j}\,.
$$
However,  for $f_{j, k, \delta_j, n_j, t}$,  we only have that, when $f_j=\Id_{F_j}$, 
\begin{equation*}
 \big| f_{j, k, \delta_j, n_j, t}(x)\Id_{\Omega^c} \big| \lesssim  \min\{1, 2^{n_j}|F_j|\} \,,
\end{equation*}
which constitutes the main obstruction to obtaining $r\leq r_J$. \\

The bilinear case of \eqref{r_Jest1}, corresponding to $n=2$, was established in \cite{Li2}. The multilinear case, corresponding to $n\geq 3$, follows by the same argument. Indeed, the decomposition, localization, and summation arguments used by the second author in \cite{Li2} remain valid in the present setting, and the additional parameters only affect the notation and the corresponding bookkeeping. In particular, all the estimates required in the bilinear argument hold uniformly with respect to these parameters, and the same summation argument yields 
\eqref{r_Jest1}.
Since reproducing the proof would amount to repeating the lengthy technical argument of \cite{Li2} without introducing any new ideas, we omit the details. We refer the reader to the proof of the bilinear estimate in \cite{Li2} for the underlying argument; the passage to the present setting is obtained by applying the same estimates and summation procedure.\\

The desired estimates for the paraproduct \(\Pi_{\delta_1,\dots,\delta_J}\) now follow by interpolation between \eqref{bana-J>2} and \eqref{r_J-est}. This completes the proof of Theorem \ref{para}. \\

The quasi-Banach case is substantially more difficult than the Banach case. The main difficulty arises from the presence of an integral involving \(\rho^*\) in the parameter \(t\). In the Banach range, Minkowski's inequality allows us to interchange the integral in \(t\) with the relevant \(L^r\)-norm, thereby avoiding this difficulty. In the quasi-Banach range, however, Minkowski's inequality is no longer available, and controlling this integral becomes a significant issue that requires a delicate technical argument.

\subsection{Alternative ways to handle the translated paraproducts}

There are at least two other ways to finish the proof of Theorem \ref{para} in the quasi-Banach range.\\

\noindent{\bf  The first way}. 
By repeating the argument for the bilinear paraproducts in \cite{Li2}, with a minor modification, we obtain the following proposition.

\begin{proposition}
Let
\begin{equation*}
 \Pi^{(2)}_{\delta_{i_1},\delta_{i_2}} (f_1, f_2)(x) = \sum_{k\in\mathbb K} \int\! |\rho^*(t)| \prod_{j=1}^2\!
 \big|f_{j, k, \delta_{i_j}, n_j, t}(x) \big|\, \mathrm{d}t \,.
\end{equation*}
Then for any $1/2<r<1$ and $p_1, p_2>1$ with $1/r=1/p_1+1/p_2$, we have 
\begin{equation*}
    \big\| \Pi^{(2)}_{\delta_{i_1},\delta_{i_2}} (f_1, f_2) \big\|_r \lesssim_{\e} 2^{\e\max\{n_1, n_2\}} \|f_1\|_{p_1}
    \|f_2\|_{p_2}\,,
\end{equation*}
for any $f_1\in L^{p_1}$ and $f_2\in L^{p_2}$. 
\end{proposition}

Using maximal function theory,  H\"{o}lder's inequality and multilinear interpolation, we immediately obtain the following corollary, from which 
Theorem \ref{para} follows. 

\begin{corollary}
 Let $J\in\{2, \dots, n\}$.  Let $i_1, i_2\in \{1, \dots , J\}$ be any two distinct indices. 
Then for any $\e>0$,
\begin{equation*}
 \big\| \Pi_{\delta_1, \dots, \delta_J} (f_1, \dots, f_n)\big\|_{r} \lesssim 2^{\e n_J^*}\big\|f_{i_1}\big\|_{p_{i_1}}\big\|f_{i_2}\big\|_{p_{i_2}}\prod_{\substack{1\leq j\leq J\\ j\neq i_1, i_2}}
 \big\|f_j\big\|_{\infty}
 \!\!
 \prod_{j=J+1}^n \!\big\| f_j\big\|_{p_j}
\end{equation*}
holds for all $p_{i_1}, p_{i_2}, p_{J+1}, \dots, p_n>1$ and $r>r_J$ such that 
$$
\frac{1}{r}= \frac{1}{p_{i_1}} +\frac{1}{p_{i_2}}+\sum_{j=J+1}^n\frac{1}{p_j}\,. 
$$
\end{corollary}

\vspace{0.4cm}

\noindent{\bf  The second way}. 
We take \(\Pi_{\delta_1,\ldots,\delta_n}\) as an illustrative example, since the other cases can be treated similarly. For a sequence of functions \(\{f_k\}\), write \(\|\{f_k\}\|_{L^p(\ell^q)} := \|\|\{f_k\}\|_{\ell^q}\|_{L^p}\). 

First, the boundedness of the bilinear maximal operator can be used, via interpolation, to obtain the boundedness of the \(n\)-linear maximal operator. Then one constructs the vector-valued operator corresponding to \(\Pi_{\delta_1,\ldots,\delta_n}\):
\[
(\{f_{1,k}\},\dots,\{f_{n,k}\}) \mapsto \left\{ \int \prod_{1\leq j\leq n} f_{j,k,\delta_j,n_j,t}(x)\,\rho^*(t)\,dt \right\}.
\]
Here $f_{j,k,\delta_j,n_j,t}$ is defined as before, with the input $f_j$ replaced by the sequence element $f_{j,k}$. For any $(1/p_1, \dots, 1/p_n, 1/r)\in\mathbb P_n$, the Sobolev estimate (Proposition \ref{prop1}) yields the boundedness
\[
L^{p_1}(\ell^{p_1}) \times \dots \times L^{p_n}(\ell^{p_n}) \to L^r(\ell^r),
\]
while the boundedness of the maximal operator yields the boundedness 
\[
L^{p_1}(\ell^\infty) \times \dots \times L^{p_n}(\ell^\infty) \to L^r(\ell^\infty).
\]
By using interpolation and the Littlewood--Paley theorem, the desired result follows. \\


\appendix

\section{Sharpness of Theorem \ref{thm-al}} \label{appendixA}

We investigate the sharpness of Theorems \ref{thm1/2} and \ref{thm-al} using a counterexample which was discovered by the authors, inspired by AI-generated suggestions.

\begin{proposition}
Let $n\geq3$, and let $m_1,\ldots,m_n$ be distinct positive integers, at least one of which is odd. Consider the operator
	\[
	T(f_1,\ldots,f_n)(x)=\int_{\mathbb{R}} \prod_{1\leq i\leq n}f_i\left(x-\frac{t^{m_i}}{m_i}\right)\rho(t)\,\mathrm{d}t,
	\]
	where \(\rho\) is a smooth function supported in \(\{t\in\mathbb{R}  : |t|\sim 1\}\) such that $\rho(1)\neq 0$. For any exponents $p_{1}, \dots,p_{n},r$ satisfying
	\begin{equation*}
		\frac{1}{n}<r<\frac12,\quad\ 1<p_i<\infty,\quad \frac{1}{r}=\sum_{1\le i\le n} \frac{1}{p_i},
	\end{equation*}
	the inequality
	\[
	\|T(f_1,\ldots,f_n)\|_r\lesssim \prod_{1\leq i\leq n}\|f_i\|_{p_i}
	\]
    cannot hold for all functions $f_i\in L^{p_i}(\mathbb R)$.
	\end{proposition}

	\begin{proof}
Take a smooth function \(\chi\) such that
\[
0\leq \chi \leq 1,\qquad \chi = 1 \text{ on }[-1,1],\qquad \operatorname{supp}\chi \subset [-2,2].
\] 
 For a sufficiently large \(N\), let
		\[
		f_{i}(y) = \chi \left(N^{2}\left(y + \frac{1}{m_i}\right)\right), \qquad 1\leq i\leq n.
		\]
		Then
		\[
		\prod_{1\leq i\leq n}\| f_{i}\|_{p_{i}}\lesssim N^{-\frac{2}{r}}.
		\]
		To obtain a lower bound for $\|T(f_1,\ldots,f_n)\|_r$, we only consider the range \(|x|\leq c_1 N^{-1}\), where $c_1>0$ will be chosen sufficiently small. Choose an index $\ell$ for which $m_\ell$ is odd. If the integrand is nonzero, then
\[
 \left|x-\frac{t^{m_\ell}}{m_\ell}+\frac1{m_\ell}\right|
 \leq2N^{-2},
 \qquad
 |t^{m_\ell}-1|
 \leq m_\ell\bigl(|x|+2N^{-2}\bigr).
\]
Since $m_\ell$ is odd, for all sufficiently large $N$, we have
\begin{equation*}
 |t-1|\lesssim_{m_\ell,c_1}N^{-1}.
\end{equation*}
In particular, there is no contribution from a neighborhood of $t=-1$.

       If $|t-(1 + x)|\leq 0.5 N^{-2}$, then, by Taylor's expansion of $t^{m_i}$ at $t=1$, for $c_1$ sufficiently small and  $N$ sufficiently
large,  
		\[
		 \left|x-\frac{t^{m_i}}{m_i}+\frac1{m_i}\right|\leq\frac{1}{2}N^{-2}+C_{m_i}(t-1)^2\leq N^{-2}, \qquad 1\leq i\leq n.
		\]
		Hence every factor $f_i(x-t^{m_i}/m_i)$ equals $1$ on this range of \(t\), and the integrand is close to \(\rho(1)\). Therefore,
		\[
		\left\|\int \prod_{1\leq i\leq n}f_{i}\left(x - \frac{t^{m_i}}{m_i}\right)\rho(t)\,\mathrm{d}t\right\|_{r}\gtrsim N^{-2 - \frac{1}{r}}.
		\]
		Letting \(N\rightarrow \infty\) shows that $r\geq1/2$ is necessary, which completes the proof.
	\end{proof}
	
	The same construction gives the corresponding counterexample for the multilinear Hilbert transform.

		\begin{proposition}
Let $n\geq3$, and let $m_1,\ldots,m_n$ be distinct positive integers, at least one of which is odd. For any exponents $p_{1}, \dots,p_{n},r$ satisfying
	\begin{equation*} 
		\frac{1}{n}<r<\frac12,\quad\ 1<p_i<\infty,\quad \frac{1}{r}=\sum_{1\le i\le n}\frac{1}{p_i},
	\end{equation*}
	the inequality
		\[
		\left\| \mathrm{p.v.}\int_{\mathbb{R}} \prod_{1\leq i\leq n}f_{i}\left(x - \frac{t^{m_i}}{m_i}\right)\frac{\mathrm{d}t}{t}\right\|_{r}\lesssim \prod_{1\leq i\leq n}\| f_{i}\|_{p_{i}}
		\]
         cannot hold for all functions $f_i\in L^{p_i}(\mathbb R)$.
		\end{proposition}

	It is easy to observe that the counterexample above also applies to the multilinear maximal function along $(t^{m_1}/m_1,\ldots,t^{m_n}/m_n)$. The assumption that at least one exponent is odd is not necessary, because  the maximal operator has a nonnegative integrand, so restricting the integral to a neighborhood of \(t=1\) gives the required lower bound regardless of the parity of the degrees.


	\section{Sharpness of Theorem \ref{thm}} \label{appendixB}

        We investigate the sharpness of Theorem \ref{thm} using a counterexample suggested by an AI tool and subsequently verified by the authors. The investigation began with the special  moment curve $t, t^2, t^3$.

		\begin{proposition}\label{THT}
		If the exponents \(p_1, p_2, p_3, r\) satisfy
		\[
		1 < p_1, p_2, p_3 < \infty, \quad r^{-1} = \sum_{1\leq i\leq 3} p_i^{-1}, \quad r^{-1} + 2p_1^{-1} > 4,
		\]
		then the inequality
		\[
		\left\| \mathrm{p.v.} \int_{\mathbb{R}} f_1(x - t)f_2(x - t^2)f_3(x - t^3) \frac{\mathrm{d}t}{t} \right\|_r \lesssim \prod_{1\leq i\leq 3} \|f_i\|_{p_i}
		\]
        cannot hold for all functions $f_i\in L^{p_i}$.
	\end{proposition}
		
		\begin{remark}
			When \(1/3 < r < 5/12\), the exponents \(p_1^{-1} = p_2^{-1} = p_3^{-1} = (3r)^{-1}\) satisfy the conditions in the above proposition.  When \(4/11 < r < 1/2\), the choices
		\[
		   p_1^{-1} = 1 - \epsilon, \quad p_2^{-1} = p_3^{-1} = \frac{1}{2} + 2\epsilon, \quad \textrm{with}\ \epsilon = \frac{r^{-1} - 2}{3} \in \left(0, \frac{1}{4}\right),
		\]
        also satisfy the conditions.
		This shows that for any \(1/3 < r < 1/2\), there exist exponents $p_{1},p_{2},p_{3}$ satisfying the H\"older relation such that the desired $(p_1, p_2, p_3, r)$-boundedness fails. 
For a fixed $r\in(1/3,1/2)$, it remains unknown whether there exist other choices of exponents for which the desired boundedness holds.
		\end{remark}
		
		In the following, let \(a = 1/4\), and \(\Phi_1, \Phi_2, \Phi_3\) be three one-variable functions defined on a small neighborhood of \(0\), determined by
        \[
			\Phi_1(x) = x, \qquad \Phi_2(t - t^2 - a + a^2) = -\Theta(t), \qquad \Phi_3(t - t^3 - a + a^3) = \Theta(t),
			\]
        where
        \begin{equation*}
        \Theta(t) = \int_a^t \frac{(1 - 2u)(1 - 3u^2)}{u(2 - 3u)} \, \mathrm{d}u
        \end{equation*}
        is defined on a small neighborhood of $a$.	By the inverse function theorem, the map
$$ (x,t)\mapsto \bigl(\Phi_1(x-t+a),\,\Phi_2(x-t^2+a^2)\bigr) $$
is a smooth diffeomorphism from a sufficiently small neighborhood of \((0,a)\) onto a neighborhood of \((0,0)\). We work in these neighborhoods below.

		Let $N$ be a large integer, $M=\lfloor\eta N^{3}\rfloor$, and $s\in[0,\eta]$,  where $\eta>0$ is a small absolute constant to be determined later. We consider a grid partition near \((0, 0)\) of mesh size \(N^{-3}\) with translation \((0, s)\) (referred to as ``grid \(1\)''); that is, we consider the nodes
		\[
		\left( i N^{-3}, s+j N^{-3}\right), \qquad 1 \leq i\leq N, 1\leq j \leq M.
		\]
		By the inverse mapping, we obtain a grid partition near \((0, a)\) (referred to as ``grid \(2\)''), whose nodes are the preimages of the above nodes, denoted by \(\left(x_{ij}(s), t_{ij}(s)\right)\) (for brevity, we write \(x_{ij}, t_{ij}\)). Thus, 
		\[
		\Phi_{1}\left(x_{ij} - t_{ij} + a\right) = i N^{-3},
		\]
		\[
		\Phi_{2}\left(x_{ij} - t_{ij}^{2} + a^{2}\right) =s+ j N^{-3}.
		\]

        We observe that the points of grid \(1\) are too regularly aligned: the projections onto the \(x\)-axis of the radius-\(N^{-4}\) neighborhoods centered at its nodes overlap heavily, so their union has measure only \(O(N^{-3})\). By contrast, in grid \(2\), by choosing the translation parameter \(s\) appropriately, we can stagger these projections so that their union has substantially larger measure.
        The following lemma makes this precise.

		\begin{lemma}\label{lemma1}
        With the notation above, define
        \[
		E_{M,N}:=\bigcup_{1\leq i\leq N,1\leq j\leq M}E_{M,N}^{(i,j)}:= \bigcup_{1\leq i\leq N,1\leq j\leq M}\left[x_{ij}(s) - N^{-4},x_{ij}(s) + N^{-4}\right].
		\]
        Then there exists \(s\in[0,\eta]\) such that
        $$ |E_{M,N}(s)|\gtrsim\eta, $$
        with an implicit constant independent of \(M\), \(N\), and \(\eta\).
		\end{lemma}
		
		\begin{proof}
		In what follows, we omit the  ranges of \(i,j,i',j'\), and assume  by default that \(1\leq i,i'\leq N,\ 1\leq j,j'\leq M\). 
		By H\"older's inequality,
		\[
		|E_{M,N}|\geq \frac{(\int\sum_{i,j}1_{E_{M,N}^{(i,j)}})^{2}}{\int(\sum_{i,j}1_{E_{M,N}^{(i,j)}})^{2}} \geq \frac{\eta^{2}}{\sum_{i,i',j,j'}|E_{M,N}^{(i,j)}\cap E_{M,N}^{(i',j')}|}.
		\]
		Let \(Q_{M,N} = Q_{M,N}(s)\) be the number of quadruples \((i,i',j,j')\) such that \(E_{M,N}^{(i,j)}\cap E_{M,N}^{(i',j')}\neq \emptyset\). It suffices to prove that there exists \(0\leq s\leq \eta\) such that \(Q_{M,N}(s)\lesssim NM\). This reduces to showing that the integral average of \(Q_{M,N}\) over \(\left[0,\eta\right]\) is \(\lesssim NM\). 
        
        By Fubini’s theorem and distinguishing the diagonal and off-diagonal cases, we obtain
\begin{align*}
			&\eta^{-1}\int_{0}^{\eta}Q_{M,N}(s)\,\mathrm{d}s\\
            = &\eta^{-1}\sum_{i,i',j,j'}\left|\left\{s\in \left[0,\eta\right] : |x_{ij}(s) - x_{i'j'}(s)|\leq 2N^{-4}\right\}\right|\\
		= &NM + \eta^{-1}\sum_{\substack{i',j'\\1\leq K\leq M}}\quad \sum_{\substack{i,j\\ \max(|i-i'|,|j-j'|)=K}}\left|\left\{s\in [0,\eta] : |F(s)|\leq 2N^{-4}\right\}\right|,
	\end{align*}
		where \(F(s):= x_{ij}(s) - x_{i'j'}(s)\). It suffices to show that the off-diagonal part is of size \(O (NM)\) whenever $\eta$ is sufficiently small and $N$ is sufficiently large. This follows easily from the following claims: 
        for fixed $i'$, $j'$, and $K$, if 
		\[\left\{s\in [0,\eta] : |F(s)|\leq 2N^{-4}\right\}\ne\emptyset,
		\]
		then \(F\) is monotone on \([0,\eta]\) and $|F'| \gtrsim N^{-3}K$ on $[0,\eta]$; only the case $K\lesssim N$ needs to be counted;  moreover, the number of pairs \((i,j)\) satisfying \(K = \max \{|i - i'|,|j - j'|\}\) is \(O(\eta K+1)\).

        It remains to prove these claims. Applying the fundamental theorem of calculus along the segment joining the two grid points, followed by Taylor estimates near the origin, we obtain the following expansions, uniformly for \(s\in[0,\eta]\):
		\[
		N^{3}F(s) = \lambda_{1}(i - i^{\prime})+\lambda_{2}(j - j^{\prime}) + O\left(\eta K\right),
		\]
		\[
		N^{3}F^{\prime}(s) = \mu_{1}(i - i^{\prime}) + \mu_{2}(j - j^{\prime}) + O\left(\eta K\right),
		\]
		where the matrix
		\[
		\begin{bmatrix}
			\lambda_{1} & \lambda_{2}\\
			\mu_{1} & \mu_{2}
		\end{bmatrix}=\begin{bmatrix}
			-1 & -10/13\\
			80/13 & 9680/2197
		\end{bmatrix}
		\]
		is invertible, which leads to 
		\begin{equation*} 
		K\asymp |\lambda_{1}(i - i^{\prime}) + \lambda_{2}(j - j^{\prime})| + |\mu_{1}(i - i^{\prime}) + \mu_{2}(j - j^{\prime})|.
	\end{equation*}

		Since $|F(s_{0})|\leq 2N^{-4}$ for some $s_{0}$ by assumption, we have
		\begin{equation}\label{eq1}
			|\lambda_{1}(i - i^{\prime}) + \lambda_{2}(j - j^{\prime})|\lesssim N^{-1}+\eta K\lesssim \eta K.
		\end{equation}
        Combining the preceding two equations shows that, if $\eta$ is sufficiently small, then 
     \[
	|\mu_1(i - i') + \mu_2(j - j')| \gtrsim K, 
	\]
    and hence 
    \[
      |F'| \gtrsim N^{-3}K \text{ on } [0,\eta].
     \]
     The monotonicity of $F$ then follows readily.
     
	When \(K \geq 3N\), since $|i-i'|<N$, we have $K=|j-j'|$. Therefore,  
\[
	|\lambda_1(i - i') + \lambda_2(j - j')| \geq |\lambda_2|K - |\lambda_1|N\gtrsim K.
	\]
	By shrinking \(\eta\) if necessary, this contradicts \eqref{eq1}.  This shows that we only need to consider the case \(K \leq 3N\).

To count the number of elements in the set    
	\begin{equation*}
		\{(i,j): K = \max(|i - i'|, |j - j'|)\},
	\end{equation*}
    we first observe that we must have $|j - j'|=K$. Indeed, if $|i - i'|=K$, then
    \begin{equation*}
        |\lambda_{1}(i - i^{\prime}) + \lambda_{2}(j - j^{\prime})|\geq |\lambda_{1}|K-|\lambda_{2}|K\gtrsim K,
    \end{equation*}
	which contradicts \eqref{eq1}. When $j-j'=K$, \eqref{eq1} implies that the number of possible values of $i$  is $O(\eta K+1)$. The case $j-j'=-K$ is similar. Hence, the total number of elements is $O(\eta K+1)$.
	\end{proof}

	We can now complete the construction of the counterexample.

	\begin{proof}[Proof of Proposition \ref{THT}]
			Choose $\chi,\rho\in C_c^\infty(\mathbb R)$ with the following properties:  \(0\leq \chi \leq 1\), \(\chi\) is identically \(1\) near \(0\), and \(\operatorname{supp} \chi\) is sufficiently small so that \(|\Phi_{i}^{\prime}|\sim 1\) on \(\operatorname{supp} \chi\);	\(0\leq \rho\leq 1\), \(\rho\) is identically \(1\) on \([- 1,1]\), and vanishes outside \([- 2,2]\).
		
		Let \(s\) be the translation parameter given by Lemma \ref{lemma1}. Construct three nonnegative functions
		\[
		f_{1}(y) = \chi (y + a)\sum_{1\leq i\leq N}\rho \left(\frac{\Phi_{1}(y + a) - iN^{-3}}{AN^{-4}}\right),
		\]
		\[
		f_{2}(y) = \chi (y + a^{2})\sum_{1\leq j\leq M}\rho \left(\frac{\Phi_{2}(y + a^{2}) -s- jN^{-3}}{AN^{-4}}\right),
		\]
		\[
		f_{3}(y) = \chi (y + a^{3})\sum_{2\leq k\leq N+M}\rho \left(\frac{\Phi_{3}(y + a^{3}) + s + kN^{-3}}{AN^{-4}}\right),
		\]
		where \(A > 0\) is a large absolute constant to be determined later. Note that (taking \(f_{1}\) as an example) for every \(i\),
		
		\[\left\{ y : \left| \frac{\Phi_1(y+a) -  iN^{-3}}{A N^{-4}} \right| \leq 2 \right\}\]
		is contained in the interval centered at \(\Phi_{1}^{-1}\left(iN^{-3}\right) - a\) with radius \(\lesssim_A N^{- 4}\). Also, since the distance between the centers corresponding to different \(i\) is \(\gtrsim N^{- 3}\), we have \(0\leq f_{1}\leq 1\) and \(|\operatorname{supp} f_{1}|\lesssim N^{- 3}\), hence \(\| f_{1}\|_{p_{1}}\lesssim N^{-3/p_1}\). The same argument applies to $f_2$ and $f_3$. Therefore,
		\[
		\prod_{1\leq i\leq 3}\| f_{i}\|_{p_{i}}\lesssim N^{-\frac{3}{p_1}-\frac{1}{p_2}-\frac{1}{p_3}}= N^{-\frac{2}{p_1}-\frac{1}{r}}.
		\]
		
		We  consider the principal value integral
		\[
		\mathrm{p.v.}\int_{\mathbb{R}} f_{1}(x - t)f_{2}(x - t^{2})f_{3}(x - t^{3})\frac{\mathrm{d} t}{t}
		\]
		on the range \(x \in E_{M,N}\) (in which case we have \(|x| \lesssim \eta \)). Since \(f_{1}(x-t)\ne 0\) implies \(t\sim a\) provided $\eta$ is sufficiently small,  the integral is supported away from \(t=0\), and no principal value is needed.  The condition \(x \in E_{M,N}\) implies that \(x\) lies within a distance of \(N^{-4}\) of some \(x_{ij}\).  We now fix \(x\), \(i\), and \(j\), and restrict attention to \(t\) satisfying \(|t-t_{ij}|\le N^{-4}\). If \(A\) is sufficiently large, then one can guarantee that
		\[
		f_{1}(x - t) = \chi (x - t + a)\rho \left(\frac{\Phi_{1}(x - t + a) - \Phi_{1}(x_{ij} - t_{ij} + a)}{AN^{-4}}\right) \equiv 1.
		\]
		Similarly, \(f_{2}(x - t^{2}) \equiv 1\) provided $\eta$ is sufficiently small. 
        
        For $|x-x_{ij}|\le N^{-4}$ and $|t-t_{ij}|\le N^{-4}$,
       Taylor's theorem gives
		\begin{equation*}
			|\Phi_1(x - t + a) + \Phi_2(x - t^2 + a^2) + \Phi_3(x - t^3 + a^3)|\lesssim |x-t+a|^2 \lesssim N^{-4}. 
		\end{equation*}
		Thus,
		\[
		\left|s + (i + j)N^{-3} + \Phi_3(x_{ij} - t_{ij}^3 + a^3)\right| \lesssim N^{-4}.
		\]
		Therefore, if \(A\) is sufficiently large and $\eta$ is sufficiently small, we have   
		\[
		f_3(x - t^3) = \chi(x - t^3 + a^3)\rho \left( \frac{\Phi_3(x - t^3 + a^3) + s + (i + j)N^{-3}}{A N^{-4}} \right) \equiv 1.
		\]
		Thus
		\[
		\left\| \mathrm{p.v.}\int f_{1}(x - t)f_{2}(x - t^{2})f_{3}(x - t^{3})\frac{\mathrm{d}t}{t}\right\|_{r} \gtrsim N^{-4}|E_{M, N}|^{1/r} \gtrsim N^{-4}\eta^3\gtrsim N^{-4}.
		\]
		Letting \(N \to \infty\) completes the proof.
		\end{proof}

	The construction above can be easily extended to the $n$-linear case with $n\geq 4$ involving the moment curve $t, t^2, \dots, t^n$.
    
		\begin{proposition}\label{MHT}
			If the exponents \(p_1, \dots, p_n, r\) satisfy
			\[
			1 < p_1, \ldots, p_n < \infty, \qquad r^{-1} = \sum_{1\leq i\leq n} p_i^{-1}, \qquad 3p^{-1}_1+p_2^{-1}+p_3^{-1}> 4,
			\]
			then the inequality
			\[
			\left\| \mathrm{p.v.} \int_{\mathbb{R}} \prod_{1\leq i\leq n}f_i(x - t^{i}) \frac{\mathrm{d}t}{t} \right\|_r \lesssim \prod_{1\leq i\leq n} \|f_i\|_{p_i}
			\]
            cannot hold for all functions $f_i\in L^{p_i}$.
		\end{proposition}

		\begin{remark}
			For $n\geq 4$, given \(1/n < r < 1/2\), define
			\[
			L_0 = \max\{2, r^{-1} - (n - 3)\}, \quad B_0 = \min\{3, r^{-1}\}, \quad V = \frac{L_0 + B_0}{2}.
			\]
			We have \(L_0 < B_0\).  Hence
			\[
			2 < V < 3, \quad 0 <r^{-1} - V < n - 3.
			\]
			Let
			\[
			\e = \min\left\{\frac{V - 2}{4}, \frac{3 - V}{2}\right\} > 0.
			\]
			It is easy to check that the exponents
			\[
			p_1^{-1} = 1 - \e, \quad p_2^{-1} = p_3^{-1} = \frac{V - 1 + \e}{2}, \quad p_4^{-1} = \cdots = p_n^{-1} = \frac{r^{-1} - V}{n - 3}
			\]
			satisfy the conditions of Proposition \ref{MHT}. This shows that for any \(1/n < r < 1/2\), there exist exponents $p_{1},\dots,p_{n}$ satisfying the H\"older relation such that the desired $(p_1, \dots, p_n, r)$-boundedness fails.
		\end{remark}

		\begin{proof}[Proof of Proposition \ref{MHT}]
			We follow the construction from the trilinear case above, keeping the definitions of \(f_1, f_2, f_3\) unchanged, and additionally define \(f_i\) for \(i \geq 4\) as follows:
            \[f_{i}(x):=\chi\left(\frac{x+a^{i}}{A_{i}\eta}\right).
			\]
			In the lower-bound argument, we in fact have $|x|\lesssim \eta$ and $|t-a|\lesssim \eta$. Hence, for all $i\geq 4$, if $A_{i}$ is sufficiently large, then $f_{i}(x-t^{i})\equiv 1$. The remaining estimates are unchanged. 
		\end{proof}

The construction above also applies to the curves
$(t^{m_1},\ldots,t^{m_n})$, where $n\geq 3$ and the exponents
$m_1,\ldots,m_n$ are distinct positive integers, at least
one of which is odd, with the same exponent conditions as in
Proposition \ref{MHT}.

To see this, write $P_j(t)=t^{m_j}$ and $d_j(t)=P_j'(t)$.
Replace the fixed base point $a=1/4$ by a point $a>0$ at which
$d_1(a),d_2(a),d_3(a)$ are distinct and $R'(a)\neq 0$,
where
\[
R(t)=\frac{d_1(t)(d_2(t)-d_3(t))}
           {d_2(t)(d_1(t)-d_3(t))}.
\]
Define $\Theta$ near $a$ by
\[
\Theta(a)=0,
\qquad
\Theta'(t)=
\frac{(d_1(t)-d_2(t))(d_1(t)-d_3(t))}
     {d_2(t)-d_3(t)},
\]
and define the local phase functions by
\[
\Phi_1(z)=z,
\qquad
\Phi_j\bigl(P_1(t)-P_j(t)-P_1(a)+P_j(a)\bigr)
   =(-1)^{j+1}\Theta(t),
\quad j=2,3,
\]
whose
sum still vanishes to second order along the corresponding
curve. The determinant in the covering argument is still nonzero.  Counting the relevant lattice points on all four
sides of the corresponding square $\max\{|i-i'|,|j-j'|\}=K$ yields the same covering estimate. Thus the grid scales and the resulting norm estimates remain
unchanged.
For $n>3$, the additional input functions are fixed smooth cutoffs equal
to one near $-P_j(a)$.
Finally, the input function corresponding to an odd exponent localizes the full integrand near the positive base
point, preventing cancellation from negative $t$. Hence the original
failure condition $3/p_1+1/p_2+1/p_3>4$ remains valid.


\subsection*{Acknowledgments}
 J. Guo was partially supported by NSFC Grants 12571110 and 12341102. X. Li was partially supported by the Simons Foundation Grant 00958361 and by NSF Grant 2350101. J. Guo would like to thank Prof. Renhui Wan for stimulating discussions during the conference at Chongli, China.  G. Zhan thanks Chiyu Zhou for helpful communications.


\end{document}